%% file: IPCO2025-long.tex
\documentclass[11pt,letterpaper]{article}
\usepackage[margin=1.3in]{geometry}
\usepackage[T1]{fontenc}
\usepackage{graphicx}
\usepackage{tikz}

\usepackage{enumitem} 
\setlist[itemize]{noitemsep,nolistsep} 
\setlist[enumerate]{noitemsep,nolistsep} 

\usepackage{mathtools}
\usepackage{amsfonts}
\usepackage{xifthen}
\usepackage[unicode]{hyperref}
\usepackage[capitalize,noabbrev]{cleveref}
\usepackage{enumitem}

\usepackage{subfigure}
\usepackage{longtable,booktabs,array}

 \usepackage[usenames,dvipsnames]{pstricks}
 \usepackage{epsfig}
 \usepackage{pst-grad} 
 \usepackage{pst-plot} 
 \usepackage[space]{grffile} 
 \usepackage{etoolbox} 
 \makeatletter 
 \patchcmd\Gread@eps{\@inputcheck#1 }{\@inputcheck"#1"\relax}{}{}
 \makeatother

\usepackage{amsthm}
\newtheorem{theorem}{Theorem}

\newtheorem{lemma}{Lemma}
\newtheorem{corollary}{Corollary}

\newtheorem{observation}{Observation}
\newtheorem{example}{Example}

\DeclareMathOperator{\poly}{poly}
\DeclareMathOperator{\size}{size}

\newcommand{\BPOL}{BPO$_\text{L}$}

\usepackage{color}

\NewDocumentCommand{\from}{}{\ensuremath{\colon}}
\NewDocumentCommand{\st}{}{\ensuremath{s.t. }}
\RenewDocumentCommand{\to}{}{\ensuremath{\rightarrow}}
\NewDocumentCommand{\NP}{}{\ensuremath{\mathsf{NP}}}
\RenewDocumentCommand{\P}{}{\ensuremath{\mathsf{P}}}
\NewDocumentCommand{\argmax}{}{\ensuremath{\text{argmax}}}
\DeclareMathOperator{\cert}{cert}
\NewDocumentCommand{\y}{}{\ensuremath{\mathbf{y}}}
\NewDocumentCommand{\x}{}{\ensuremath{\mathbf{x}}}

\NewDocumentCommand{\Z}{}{\ensuremath{\mathbb{Z}}}
\NewDocumentCommand{\Q}{}{\ensuremath{\mathbb{Q}}}
\NewDocumentCommand{\R}{}{\ensuremath{\mathbb{R}}}

\NewDocumentCommand{\BPO}{O{H} O{p}}{\ensuremath{\mathsf{BPO}(#1,#2)}}

\NewDocumentCommand{\BOP}{}{Boolean Maximization Problem}

\NewDocumentCommand{\topk}{}{\ensuremath{\mathsf{top}^k}}

\NewDocumentCommand{\prim}{O{F}}{\ensuremath{\mathsf{G_{prim}}(#1)}}
\NewDocumentCommand{\inc}{O{F}}{\ensuremath{\mathsf{G_{inc}}(#1)}}
\NewDocumentCommand{\neigh}{O{v} O{H}}{\ensuremath{N_{#2}(#1)}}
\NewDocumentCommand{\oneigh}{O{v} O{H}}{\ensuremath{N^*_{#2}(#1)}}
\NewDocumentCommand{\calH}{}{\ensuremath{\mathcal{H}}}
\NewDocumentCommand{\calM}{}{\ensuremath{\mathcal{M}}}
\NewDocumentCommand{\calT}{}{\ensuremath{\mathcal{T}}}

\NewDocumentCommand{\hwidth}{m m m}{%
  \ifthenelse{\isempty{#2}}%
  {\ensuremath{\mathsf{#3}(#1)}}%
  {\ensuremath{\mathsf{#3}(#1, #2)}}%
}
\NewDocumentCommand{\ptw}{m O{}}{\hwidth{#1}{#2}{ptw}}
\NewDocumentCommand{\itw}{m O{}}{\hwidth{#1}{#2}{itw}}
\NewDocumentCommand{\tw}{m O{}}{\hwidth{#1}{#2}{tw}}
\NewDocumentCommand{\mimw}{m O{}}{\hwidth{#1}{#2}{mimw}}
\NewDocumentCommand{\icw}{m O{}}{\hwidth{#1}{#2}{icw}}

\NewDocumentCommand{\out}{O{C}}{\ensuremath{\mathsf{out}}(#1)}
\NewDocumentCommand{\inputs}{O{v}}{\ensuremath{\mathsf{inputs}}(#1)}
\NewDocumentCommand{\var}{m}{\ensuremath{\mathsf{var}(#1)}}

\NewDocumentCommand{\ine}{O{g}}{\ensuremath{\mathsf{in_E}}(#1)}
\NewDocumentCommand{\oute}{O{g}}{\ensuremath{\mathsf{out_E}}(#1)}
\NewDocumentCommand{\ing}{O{g}}{\ensuremath{\mathsf{in}}(#1)}
\NewDocumentCommand{\outg}{O{g}}{\ensuremath{\mathsf{out}}(#1)}
\NewDocumentCommand{\edges}{O{C} O{}}{\ensuremath{\mathsf{edges}_{#2}(#1)}}
\RenewDocumentCommand{\R}{}{\ensuremath{\mathbb{R}}}

\NewDocumentCommand{\zvar}{O{g}}{\ensuremath{\mathsf{z}}(#1)}
\newcommand{\transp}{\mathsf T}

\renewcommand{\S}{\mathcal S}
\renewcommand{\P}{\mathcal P}

\newcommand{\MP}{\mathcal{MP}}
\newcommand{\MS}{\mathcal{MS}}

\DeclareMathOperator{\conv}{conv}

\usepackage{tikz}
\usetikzlibrary{shapes.geometric, arrows}
\usetikzlibrary{positioning}

\tikzstyle{startstop} = [rectangle, rounded corners, 
minimum size=6mm, 
align=center, 
draw=black]

\begin{document}
\title{A Knowledge Compilation Take \\ on Binary Polynomial Optimization}
%
%

\author{Florent Capelli 
\thanks{Univ. Artois, CNRS, UMR 8188, Centre de Recherche en Informatique de Lens (CRIL), F-62300 Lens, France}
\and
Alberto Del Pia 
\thanks{Department of Industrial and Systems Engineering 
              \& Wisconsin Institute for Discovery, 
              University of Wisconsin-Madison, Madison, WI, USA}
\and
Silvia Di Gregorio
\thanks{Universit\'e Sorbonne Paris Nord, Laboratoire d'Informatique de Paris Nord (LIPN), CNRS UMR 7030, F-93430, Villetaneuse, France}}

\date{October 2024}

\maketitle              
\begin{abstract}
      The Binary Polynomial Optimization (BPO) problem is defined as the problem of maximizing a given polynomial function over all binary points. 
    The main contribution of this paper is to draw a novel connection between BPO and the field of Knowledge Compilation.
    This connection allows us to unify and significantly extend the state-of-the-art for BPO, both in terms of tractable classes, and in terms of existence of extended formulations.
    In particular, for instances of BPO with hypergraphs that are either $\beta$-acyclic or with bounded incidence treewidth, we obtain strongly polynomial algorithms for BPO, and extended formulations of polynomial size for the corresponding multilinear polytopes.
    The generality of our technique allows us to obtain the same type of results for extensions of BPO, where we enforce extended cardinality constraints on the set of binary points, and where variables are replaced by literals.
    We also obtain strongly polynomial algorithms for the variant of the above problems where we seek $k$ best feasible solutions, instead of only one optimal solution.
    Computational results show that the resulting algorithms can be significantly faster than current state-of-the-art.
\end{abstract}



%
%
%
\section{Introduction}

In \emph{Binary Polynomial Optimization (BPO),} we are given a hypergraph $H = (V,E)$, a \emph{profit} function $p \from E \to \Q$, 
and our goal is to find some $x$ that attains
\begin{align}
\label{pr BPO}
\tag{BPO}
\begin{split}
    \max_x \quad & 
    \sum_{e \in E} p(e) \prod_{v \in e} x(v) \\
    \text{s.t.} \quad & x \in \{0,1\}^V.
    \end{split}
\end{align}
BPO is strongly NP-hard, as even the case where all the monomials have degree equal to 2 has this property~\cite{GarJohSto}. 
Nevertheless, its generality and wide range of applications attract a significant amount of interest in the optimization community. 
A few examples of applications can be found in classic operations research, computer vision, communication engineering, and theoretical physics \cite{Ber87,BorHam02,Sch09,LieMarPagRicSch10,Ish11,POLIP}.

To the best of our knowledge, three main tractable classes of BPO have been identified so far, which depend on the structure of the hypergraph $H$.
These are instances such that:
\begin{itemize}
\item[(i)] 
The primal treewidth of $H$, denoted by $\ptw{H}$, is bounded by $\log(\poly(|V|,|E|))$;
\item[(ii)] $H$ is $\beta$-acyclic; 
\item[(iii)] $H$ is a cycle hypergraph.
\end{itemize}
The techniques used to obtain these tractability results are based either on combinatorial techniques, or on a polyhedral approach.
Combinatorial approaches directly exploit the structure of the problem, and generally result in strongly polynomial algorithms.
Algorithms of this type have been obtained in \cite{CraHanJau90} for class (i), and in 
\cite{dPDiG22SODA,dPDiG23ALG} for class (ii).
On the other hand, polyhedral approaches result in linear programming reformulations, which can be used to solve the problem in polynomial time \cite{Kha79} or strongly polynomial time \cite{Tar86}, and to obtain strong relaxations for more general problems with tractable substructures.
Polyhedral approaches are based on 
the introduction of variables $x(e)$, for every $e \in E$, which leads to the reformulation of BPO as 
\begin{align*}
    \max_x \quad & 
    \sum_{e \in E} p(e) x(e) \\
    \text{s.t.} \quad & x(e) = \prod_{v \in e} x(v) \\
    & x \in \{0,1\}^{V \cup E}.
\end{align*}
The feasible region of the above optimization problem is
\begin{equation*} 
\MS_H = \Big\{ x \in \{0,1\}^{V \cup E} : x(e) = \prod_{v \in e} x(v), \; \forall e \in E \Big\},
\end{equation*}
and is called the 
\emph{multilinear set} of $H$ \cite{dPKha17MOR}.
The convex hull of $\MS_H$ is denoted by $\MP_H$, and is called the~\emph{multilinear polytope} of $H$.
Since the objective function is linear, if a linear inequalities description of $\MP_H$
is readily available,
we can solve the original BPO problem via linear programming.
In fact, to solve the original BPO problem via linear programming, it suffices to have access to an \emph{extended formulation} of $\MP_H$, which is a linear inequalities description of a higher-dimensional polytope whose projection in the original space is $\MP_H$.
Extended formulations of polynomial size for $\MP_H$ can be found in \cite{WaiJor04,Lau09,BieMun18} for class (i), in \cite{dPKha23MPA}, for class (ii), and in \cite{dPDiG21IJO} for class (iii).
The result in \cite{dPKha23MPA} extends previous extended formulations for Berge-acylic instances \cite{dPKha18SIOPT,BucCraRod18}, for $\gamma$-acyclic instances \cite{dPKha18SIOPT}, and for kite-free $\beta$-acyclic instances \cite{dPKha21MOR}.

In this paper we provide a novel connection between BPO and the field of \emph{knowledge compilation,} which allows us to unify and significantly extend most known tractability results for BPO, including all the ones mentioned above.
We establish this connection on two different levels: (1) in terms of polynomial time algorithms, and (2) in terms of extended formulations of polynomial size.

The first contribution is to draw a connection between \emph{BPO} 
and the problem of performing \emph{algebraic model counting over the $(\max, +)$ semiring} on some Boolean function~\cite{kimmig2017algebraic}.
We associate a Boolean function and a weight function to each BPO such that optimal solutions of the original BPO are in one to one correspondence with the optimal solutions of the Boolean function. 
This Boolean function can be encoded as a Conjunctive Normal Form (CNF) formula whose underlying structure is very close to the structure of the original BPO problem, allowing us to leverage many known tractability results on CNF formulas to BPO such as bounded treewidth CNF formulas~\cite{samer2010algorithms} or $\beta$-acyclic formulas~\cite{Capelli17} (see~\cite[Chapter 3]{CapelliPhD} for a survey). 
More interestingly, the tractability of algebraic model counting for CNF formulas is sometimes proven in two steps: first, one computes a small data structure which represents every satisfying assignment of the CNF formula in a factorized yet tractable form, as in~\cite{BovaCMS15}, and for which we can find the optimal value in optimal time~\cite{kimmig2017algebraic,bourhis20topk}. 
This data structure can be seen as a syntactically restricted Boolean circuit, known as \emph{Decomposable Negation Normal Form circuits,} DNNF for short, which have extensively been studied in the field of Knowledge Compilation~\cite{DarwicheM2002}. 
Our connection actually shows that all feasible points of a tractable BPO instance and their value can be represented by a small DNNF, which can be efficiently constructed from the BPO instance itself.
This enables us to transfer other known tractability results to the setting of BPO. 
For example, by working directly on the Boolean circuit, we can not only prove that broad classes of BPO can be solved in strongly polynomial time, but also that we can solve BPO instances in these classes together with cardinality constraints or find the $k$ best solutions.

The second contribution is to draw a connection between \emph{extended formulations for BPO} and DNNF.
We show that if a multilinear set $\MS_H$ can be encoded via a DNNF circuit $C$, then we can extract from $C$ an extended formulation of the multilinear polytope $\MP_H$ of roughly the same size of $C$. Our extended formulation is very simple, and in particular all coefficients are $0, \pm 1$.
It also enjoys a nice theoretical property, since the system is \emph{totally dual integral} (TDI).
%
%
This connection allows us to obtain extended formulations of polynomial size for all the classes of BPO that we show can be solved in strongly polynomial time with the approach detailed above.
Besides the wide applicability of our result, our method inherits the flexibility of DNNF circuits.
As an example, for each multilinear set for which we obtain a polynomial-size extended formulation of the corresponding multilinear polytope, we directly obtain a polynomial-size extended formulation corresponding to the subset of the multilinear set obtained by enforcing cardinality constraints.
%
As a result, we provide a unifying approach for a large number of polynomial-size extended formulations, and at the same time significantly extend them.


%
%
In the remainder of the introduction we present our results in further detail.
We remark that these results constitute just the tip of the iceberg of what our novel connection can achieve.
We begin by stating our first result for BPO, and we refer the reader to \cref{sec: connection} for the definitions of $\beta$-acyclicity and of incidence treewidth of a hypergraph $H$, which we denote by $\itw{H}$.

\begin{theorem}[Tractability of BPO]
\label{th BPO short}
There is a strongly polynomial time algorithm to solve the BPO problem, provided that $H$ is $\beta$-acyclic, or the incidence treewidth of $H$ is bounded by $\log(\poly(|V|,|E|))$.
Furthermore, under the same assumptions, there exists a polynomial-size extended formulation for $\MP_H$.
\end{theorem}


\cref{th BPO short} implies and unifies the known tractable classes (i), (ii), (iii).
To see that (i) is implied, it suffices to note that $\ptw{H} \ge 2\itw{H}$.
We observe that this bound can be very loose, as in fact 
the hypergraph with only one edge containing all $n$ vertices has $\ptw{H} = n-1$ and $\itw{H}=1$. 
Hence, \cref{th BPO short} greatly extends (i).





The generality of our approach allows us to significantly extend \cref{th BPO short} in several directions.
The first extension enables us to consider a constrained version of BPO, with a more general feasible region.
Namely, we consider feasible regions consisting of binary points that satisfy \emph{extended cardinality constraints} of the form
\begin{align}
\label{eq ext card const}
\sum_{v \in V} x(v) \in S,
\end{align}
for some $S \subseteq [n]$ where $n=|V|,$
and where the set $S$ is given as part of the input.
%
Note that extended cardinality constraints are quite general and extend several classes of well-known inequalities, like cardinality constraints and modulo constraints.
To the best of our knowledge, there are only two classes of constrained BPO that are known to be tractable. 
In these two classes, the feasible region consists of
(a) binary points satisfying cardinality constraints $l \leq \sum_{v \in V} x_v \leq u$, and $E$ is nested \cite{CheDasGun23b} or $|E| = 2$ \cite{CheDasGun23};
or (b) binary points satisfying polynomial constraints
$$
\sum_{f \subseteq e_i} 
p^i(f) \prod_{v \in f} x(v) \ge 0, \qquad i=1,\dots,m,
$$
where, for $i=1,\dots,m$, $e_i \subseteq V$ and $p^i \from 2^{e_i} \to \Q$,
and it is assumed that the primal treewidth of $(V,E \cup \{e_1,\dots,e_m\})$ is bounded by $\log(\poly(|V|,|E|,m))$ \cite{WaiJor04,Lau09,BieMun18}.

Our second extension concerns the objective function of the problem.
Namely, we consider the objective function obtained from the one of BPO by replacing variables with literals:
\begin{align*}
    \sum_{e \in E} p(e) \prod_{v \in e} \sigma_e(v),
\end{align*}
where, for each $e \in E$, $\sigma_e$ is a mapping (given in input) with $\sigma_e(v) \in \{x(v), \ 1-x(v)\}$.
We refer to this class of problems as \emph{Binary Polynomial Optimization with Literals,} abbreviated \BPOL{}.
This optimization problem is at the heart of the area of research known as \emph{pseudo-Boolean optimization} in the literature, and we refer the reader to \cite{BorHam02} for a thorough survey.
In the unconstrained case, i.e., the feasible region consists of $\{0,1\}^V$, it is known that \BPOL{} can be solved in polynomial time if $H$ is $\beta$-acyclic \cite{Kam23}, and an extended formulation of polynomial size for the convex hull of the associated \emph{pseudo-Boolean set} is given in \cite{dPKha24MPA}.
This result,
as well as the tractability of (a) above, 
is implied by our most general statement, the extension of \cref{th BPO short}, given below.

%

\begin{theorem}[Tractability of \BPOL{}]
\label{th general}
There is a strongly polynomial time algorithm to solve \BPOL{} with extended cardinality constraints, provided that $H$ is $\beta$-acyclic, or
the incidence treewidth of $H$ is bounded by $\log(\poly(|V|,|E|))$.
Furthermore, under the same assumptions, there exists a polynomial-size extended formulation for the convex hull of the points in the pseudo-Boolean set that satisfy the extended cardinality constraints.
\end{theorem}

Our third and final extension allows us to consider a different overarching goal.
In fact, we consider the problem of finding $k$ best feasible solutions to the optimization problem, rather than only one optimal solution.
For all classes of BPO and \BPOL{} considered in \cref{th BPO short,th general}, we obtain 
polynomial time algorithms to find $k$ best feasible solutions. 
While this result also follows by combining our extended formulations with Theorem~4 in \cite{AngAhmDeyKai15}, our approach yields a direct and practical algorithm that does not rely on extended formulations.

A key feature of our approach is that it also leads to practically efficient algorithms for BPO.
Preliminary computational results show that the resulting algorithms can be significantly faster than current state-of-the-art on some structured instances.

\paragraph{Organization of the paper.} The paper is organized as follows: we first draw a connection between solving a BPO problem and finding optimal solutions of a Boolean function for a given weight function in \cref{sec:preliminary}.
We then explain in \cref{sec: connection} how the Boolean function can be encoded as a CNF formula that preserves the structure of the original BPO problem. \cref{sec:bpokc} explains how such CNF can be turned into Boolean circuits known as d-DNNF where finding optimal solutions is tractable. \cref{sec:beyond} explores generalizations of BPO and show how the d-DNNF representation allows for solving these generalizations by directly transforming the circuit. 
In \cref{sec:extend-form} we present our results about extended formulations.
Finally, \cref{sec:experiments} shows some encouraging preliminaries experiment. The structure of the paper summarizing our approach is shown in \cref{fig:overview}.




\begin{figure}
  \centering

\scalebox{0.7}{
\begin{tikzpicture}[node distance=0.7cm]
\node (BPO) [startstop] {BPO problem $(H,p)$:\\ $\sum_{e \in E} p(e)\prod_{v \in e}x(v)$};
\node (BOP) [startstop, right=of BPO] {Boolean Maximization \\ Problem $(f_H, w_p)$ \\ \cref{sec:preliminary}, \cref{thm:bpotobop}};
\node (CNF) [startstop, right=of BOP] {CNF encoding(s) \\ $F_H$ of $f_H$ \\ \cref{sec: connection}};
\node (dDNNF) [startstop, right=of CNF] {d-DNNF $C$ \\ representing $f_H$ \\ \cref{sec:bpokc}};
\node (OPT) [startstop, below=of CNF] {Optimal solution \\ extracted from $C$ \\ \cref{thm:maxplus_amc}};
\node (BEY) [startstop, below=of dDNNF, right=of OPT] {Generalization of BPO \\ by transforming $C$ \\ \cref{sec:beyond}};
\node (EXT) [startstop, below=of dDNNF, right=of BEY] {Extended formulations \\ extracted from $C$ \\ \cref{sec:extend-form}};

\draw[->]  (BPO) -- (BOP);
\draw[->]  (BOP) -- (CNF);
\draw[->]  (CNF) -- (dDNNF);
\draw[->]  (dDNNF) -- (OPT);
\draw[->]  (dDNNF) -- (BEY);
\draw[->]  (dDNNF) -- (EXT);

\end{tikzpicture}
}

\caption{Structure of the paper.}
\label{fig:overview}
\end{figure}

\section{Binary Polynomial Optimization and \BOP}
\label{sec:preliminary}

In this section, we detail a connection between solving a BPO problem and finding optimal solutions of a Boolean function for a given weight function.

\paragraph{Binary Polynomial Optimization Problem.} 
Our first goal is to redefine Binary Polynomial Optimization problems using hypergraphs.
A \emph{hypergraph} $H$ is a pair $(V,E)$, where $V$ is a finite set of \emph{vertices} $V$ and $E \subseteq 2^V$ is a set of \emph{edges.} 
Given a hypergraph $H=(V,E)$ and a profit function $p \from E \to \Q$, the \emph{Binary Polynomial Optimization} problem for $(H,p)$ 
is defined as finding an \emph{optimal solution} $x^* \in \{0,1\}^V$ 
to the following maximization problem: $\max_{x \in \{0,1\}^V} P_{(H,p)}(x)$ where $P_{(H,p)}$ is the polynomial defined as $P_{(H,p)}(x) = \sum_{e \in E} p(e)\prod_{v \in e} x(v)$. 
We denote this problem as $\BPO$.  
The value $P_{(H,p)}(x^*)$  
is called the \emph{optimal value} of $\BPO$.

\paragraph{\BOP.} 
A closely related notion is the notion of \emph{\BOP}. Let $X$ be a finite set of variables. We denote by $\{0,1\}^X$ the set of \emph{assignments} on variables $X$ to Boolean values, that is, the set of mappings $\tau \from X \to \{0,1\}$. A \emph{Boolean function} on variables $X$ is a subset of assignments $f \subseteq \{0,1\}^X$. An assignment $\tau \in f$ is said to be a \emph{satisfying assignment} of $f$.
A weight function $w$ on variables $X$ is a mapping $w: X \times \{0,1\} \to \Q$. Given $\tau \in \{0,1\}^X$, we let $w(\tau) = \sum_{x \in X} w(x, \tau(x))$. Given a Boolean function $f$ on variables $X$, we let  $w(f) = \max_{\tau \in f} w(\tau)$. That is
\[w(f) = \max_{\tau \in f} \sum_{x \in X} w(x, \tau(x)).\]
This value can be seen as a form of Algebraic Model Counting on the $(\Q \cup \{-\infty\}, \max, +, -\infty, 0)$ semiring, which is most of the time $\NP$-hard to compute but for which we know some tractable classes~\cite{kimmig2017algebraic}.
Given a Boolean function $f$ and a weight function $w$, the \emph{\BOP{}} $(f,w)$ is defined as the problem of computing an optimal solution $\tau^* \in f$, that is, an assignment $\tau^* \in \{0,1\}^X$ such that $w(\tau^*) = w(f)$.

\paragraph{BPO as \BOP.} There is a pretty straightforward connections between the BPO problem and the \BOP, which is based on the notion of multilinear sets introduced by Del Pia and Khajavirad in~\cite{dPKha17MOR}. 
Given a hypergraph $H=(V,E)$, the \emph{multilinear set $f_H$ of $H$} is the Boolean function on variables $X \cup Y$ where $X = \{x_v\}_{v \in V}$ and $Y = \{y_e\}_{e \in E}$ such that $\tau \in f_H$ if and only if for every $e \in E$, $\tau(y_e) = \prod_{v \in e} \tau(x_v)$. 

Moreover, given a profit function $p \from E \to \Q$, we define a weight function $w_p \from X \cup Y \times \{0,1\} \to \Q$ as: 
\begin{itemize}
    \item for every $v \in V$ and $b \in \{0,1\}$, $w_p(x_v, b) = 0$, 
    \item for every $e \in E$ and $b \in \{0,1\}$, $w_p(y_e, b) = b \times p(e)$.
\end{itemize}
In the next lemma, we show that the BPO instance $(H,p)$ is equivalent to the \BOP{} $(f_H, w_p)$, in the sense that any $\tau \in f_H$  can naturally be mapped to a feasible point $x_\tau \in \{0,1\}^V$ whose value is the weight of $\tau$. More formally:

\begin{lemma}
  \label{lem:bpotobop}
  For every $\tau \in f_H$, we have $P_{(H,p)}(x_\tau) = w_p(\tau)$ where the mapping $x_\tau \in \{0,1\}^V$ is defined as $x_\tau(v) = \tau(x_v)$ for every $v \in V$.
\end{lemma}
\begin{proof}
Recall that by definition, $\tau \in f_H$ if and only if $\tau(y_e) = \prod_{v \in e} \tau(x_v)$. We have:
  \begin{align*}
    w_p(\tau) & = \sum_{v \in V} w_p(x_v, \tau(x_v)) + \sum_{e \in E} w_p(y_e, \tau(y_e)) \\
              & = \sum_{v \in V} 0 + \sum_{e \in E} \tau(y_e) p(e)  \text{ by definition of $w_p$}\\
              & = \sum_{e \in E} p(e) \prod_{v\in e} \tau(x_v)   \text{ since $\tau \in f_H$}\\
              & = \sum_{e \in E} p(e) \prod_{v\in e} x_\tau(x_v)   \text{ by definition of $x_\tau$}\\
              & = P_{(H,p)}(x_\tau).
  \end{align*}
\end{proof}

\begin{theorem}
  \label{thm:bpotobop}
  Let $(H,p)$ be an instance of BPO. The set of optimal solutions of $\BPO$ is equal to $\{ x_{\tau^*} \mid \tau^* \text{ is an optimal solution of } (f_H,w_p)\}$.
\end{theorem}
\begin{proof}
  First of all, observe that given $x \in \{0,1\}^V$, there exists a unique $\tau \in f_H$ such that $x = x_{\tau}$. Indeed, we define $\tau(x_v) = x(v)$ and $\tau(y_e) = \prod_{v \in e} x(v)$. It is readily verified that $\tau \in f_H$ and that $x = x_\tau$. Let $x^*$ be an optimal solution of $\BPO$ and let $\tau$ be such that $x^* = x_\tau$. We claim that $\tau$ is an optimal solution of $(f_H,w_p)$. Indeed, let $\tau' \in f_H$. By \cref{lem:bpotobop}, $w_p(\tau') = P_{(H,p)}(x_{\tau'})$ and $P_{(H,p)}(x_{\tau'}) \leq P_{(H,p)}(x^*)$ since $x^*$ is optimal. Moreover, $P_{(H,p)}(x^*) = w_p(\tau)$ by \cref{lem:bpotobop} again. Hence $w_p(\tau') \leq w_p(\tau)$. In other words, $\tau$ is an optimal solution of $(f_H,w_p)$.
  Now let $\tau^*$ be an optimal solution of $(f_H,w_p)$. By a symmetric reasoning, we show that $x_{\tau^*}$ is an optimal solution of $\BPO$. Indeed, for every $x \in \{0,1\}^V$, if $\tau \in f_H$ is such that $x=x_\tau$; by \cref{lem:bpotobop},  we have $P_{(H,p)}(x) = w_p(x_\tau) \leq w_p(\tau^*) = P_{(H,p)}(x_{\tau^*})$.
\end{proof}

This connection is at the heart of our main technique for solving BPO problems in this paper: if $H$ is such that $w_p(f_H)$ is known to be tractable, then we can use \cref{thm:bpotobop} to leverage these
results on Boolean function directly to BPO. We study in the next section some Boolean function for which the \BOP{} is tractable.

\section{Encoding BPO as CNF formulas}\label{sec: connection}

In the previous section, we have established a relation between a BPO problem $(H,p)$ and the \BOP{} problem $(f_H, w_p)$ of finding an optimal solution to the Boolean function $f_H$ for the weight function $p$. 
This connection is merely a rewriting of the original problem and, without any detail on how $f_H$ is encoded, it does not provide any insight on the complexity of solving BPO. 
In this section, we explore encodings of $f_H$ in Conjunctive Normal Form (CNF), which is a very generic way of encoding Boolean functions. 
This is also 
the usual input of SAT solvers. We compare the structure of $H$ and the structure of our CNF encodings. In \cref{sec:bpokc}, we will use this structure to discover tractable classes of BPO.

\paragraph{CNF formulas.} Given a set of variables $X$, a literal over $X$ is an element of the form $x$ or $\neg x$ for some $x \in X$. Given an assignment $\tau \in \{0,1\}^X$, we naturally extend it to literals by defining $\tau(\neg x) = 1-\tau(x)$. A \emph{clause} $C$ over $X$ is a set of literals over $X$. We denote by $\var{C}$ the set of variables appearing in $C$. An assignment $\tau$ satisfies a clause $C$ if and only if there exists a literal $\ell \in C$ such that $\tau(\ell) = 1$. A \emph{Conjunctive Normal Form formula $F$ over $X$}, CNF formula for short, is a set of clauses over $X$. We denote by $\var{F}$ the set of variables appearing in $F$. An assignment $\tau$ satisfies $F$ if and only if it satisfies every clause in $F$. Hence a CNF formula over $X$ naturally induces a Boolean function over $X$, which consists of the set of satisfying assignments of $F$.

\paragraph{A first encoding.}



In this section, we fix a hypergraph $H=(V,E)$ and a profit function $p$. 
From \cref{thm:bpotobop}, one can think of the BPO instance $(H,p)$ as a Boolean function that expresses the structure of the polynomial objective function through constraints of the form $y_e = \prod_{v \in e} x_v$, for all $e \in E$. 
From a logical point of view, it corresponds to the logical equivalence $y_e \iff \bigwedge_{v \in e} x_v$, for all $e \in E$. In fact, given $e \in E$, $y_e = 1$ if and only if all the variables present in that specific monomial represented by $e$ have value equal to 1, i.e., $x_v = 1$ for all $v \in e$. 
Let us now consider some $e \in E$, and  see more in detail how the two implications of the above if and only if are expressed both from an integer optimization point of view and from a more logical point of view.

On the one hand, we must encode that if $x_v = 1$ for all $v \in e$, then $y_e=1$. 
In the integer optimization approach, the typical single constraint that is added in this case is $\sum_{v \in e} x_v - y_e \leq |e|-1$. 
The most straightforward way of encoding this as a CNF formula is as $y_e \vee \bigvee_{v \in e} \neg x_v$.
In fact, if $x_v = 1$ for all $v \in e$ then, for the formula to be satisfied, we need to have $y_e = 1$.

On the other hand, we also want to model that if $y_e = 1$ then each $v \in e$ must satisfy $x_v = 1$ too. 
In the optimization community this is usually done by adding the constraints $y_e \leq x_v$ for every $v \in e$. The easiest way to do with a CNF formula is by adding the clause $\neg y_e \vee x_v$ for each $v \in e$. If $y_e = 1$, then the clause can only be satisfied if $x_v = 1$.

The above discussion motivates the following definition: $F_H$ is the CNF formula on variables $X \cup Y$ 
having the following clauses for every $e \in E$:
\begin{itemize}
    \item $R_e = y_e \vee \bigvee_{v \in e} \neg x_v$,
    \item and for every $v \in e$, $L_{e,v} = \neg y_e \vee x_v$.
    \end{itemize}

    \begin{example}\label{example}
Let us consider the multilinear expression $-3x_1 x_2 x_3 + 4x_4 x_5 + 5x_2 x_3 x_4 x_5 x_6$. 
The typical hypergraph that represents it is $G = (V, E)$ with $V = \{ v_1, v_2, v_3, v_4, v_5, v_6 \}$ and $E$ containing only three edges $\{ v_1, v_2, v_3 \}$, $\{ v_4, v_5 \}$, $\{ v_2, v_3, v_4, v_5, v_6 \}$. 

Now let us construct its CNF encoding, which involves 9 variables and 13 clauses. Specifically it is: 
\begin{gather*}
(\neg y_{e_1} \vee x_{v_1}) \wedge (\neg y_{e_1} \vee x_{v_2}) \wedge (\neg y_{e_1} \vee x_{v_3}) \wedge
(\neg x_{v_1} \vee \neg x_{v_2} \vee \neg x_{v_3} \vee y_{e_1}) \wedge \\ 
(\neg y_{e_2} \vee x_{v_4}) \wedge (\neg y_{e_2} \vee x_{v_5}) \wedge
(\neg x_{v_4} \vee \neg x_{v_5} \vee y_{e_2}) \wedge \\
(\neg y_{e_3} \vee x_{v_2}) \wedge (\neg y_{e_3} \vee x_{v_3}) \wedge (\neg y_{e_3} \vee x_{v_4}) \wedge (\neg y_{e_3} \vee x_{v_5}) \wedge (\neg y_{e_3} \vee x_{v_6}) \wedge  \\
(\neg x_{v_2} \vee \neg x_{v_3} \vee \neg x_{v_4} \vee \neg x_{v_5} \vee \neg x_{v_6} \vee y_{e_3}) .
\end{gather*}
The corresponding hypergraph is depicted in Figure~\ref{fig: first}. 
\hfill $\diamond$
\end{example}

\begin{figure}
\centering
\begin{minipage}{.45\textwidth}
  \centering
  \includegraphics[width=7cm]{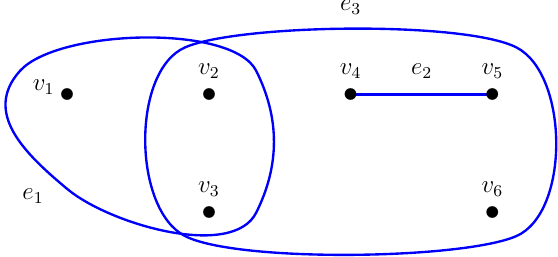}
\caption{Hypergraph $G$ representing $x_1 x_2 x_3 + x_4 x_5 + x_2 x_3 x_4 x_5 x_6$.}\label{fig: multilinear}
\end{minipage} \qquad
\begin{minipage}{.45\textwidth}
\centering
\includegraphics[width=7cm]{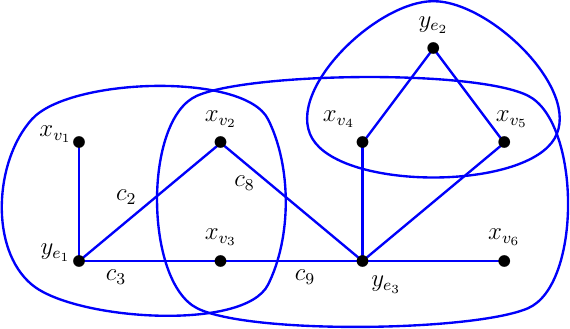}
\caption{Hypergraph representing first CNF construction.}\label{fig: first}
\end{minipage}
\end{figure}

$F_H$ is an actual encoding as witnessed by the following:    
\begin{lemma}
  \label{lem:FH_encodes_fH}
    For every hypergraph $H$, the satisfying assignments of $F_H$ are exacly $f_H$.
\end{lemma}
\begin{proof}
  $F_H$ is the conjunction of the usual clause encoding of constraints of the form $y_e \Leftrightarrow \bigwedge_{v \in e} x_v$, where $R_e$ encodes $\bigwedge_{v \in e} x_v \Rightarrow y_e$ and $L_{e,v}$ encodes $y_e \Rightarrow x_v$ and hence $\bigwedge_{v\in e} L_{e,v}$ corresponds to $y_e \Rightarrow \bigwedge_{v\in e} \tau(x_v)$.  Hence $\tau$ satisifes $F_H$ if and only if $\tau(y_e) \Leftrightarrow \bigwedge_{v \in e} \tau(x_v)$, which is equivalent to say that $\tau(y_e)=\prod_{v \in e} x_v$, that is, $\tau \in f_H$.
\end{proof}

\subsection{A CNF encoding preserving treewidth}\label{sec: treewidth}

We now turn our attention to proving that some structure of hypergraph $H$ can be found in the CNF encoding $F_H$. 
One of the most studied structures that is known to give interesting tractability both for CNF formulas and for BPO is the case of bounded treewidth hypergraphs. 
This class has a long history in the SAT solving community, where the tractability of SAT on CNF having bounded treewidth has been established in several independent work such as~\cite{alekhnovich2002satisfiability,gottlob2002fixed}, later generalized in~\cite{szeider2003fixed}. In the BPO community, several work establish the tractability for bounded treewidth~\cite{WaiJor04,Lau09,BieMun18}. 

\paragraph{Treewidth.}

Let $G=(V,E)$ be a graph. A \emph{tree decomposition $\calT=(T, (B_t)_{t \in V_T})$} is a tree $T=(V_T,E_T)$ together with a labeling $B$ of its vertices, such that the two following properties hold:
\begin{itemize}[noitemsep,nolistsep]
\item \textbf{Covering}: For every $e \in E$, there exists $t \in V_T$ such that $e \subseteq B_t$.
\item \textbf{Connectedness}: For every $v \in V$, $\{t \in V_T \mid v \in B_t\}$ is connected in $T$.
\end{itemize}
An element $B_t$ of $B$ is called a \emph{bag of $\calT$}. The \emph{treewidth of $\calT$} is defined as $\tw{G}[\calT] = \max_{t \in B_t} |B_t|-1$, that is, the size of the biggest bag minus $1$.  The \emph{treewidth of $G$} is defined to be the smallest possible width over every tree decomposition of $G$, that is, $\tw{G} = \min_{\calT} \tw{G}[\calT]$.

Two notions of treewidth can be defined for hypergraphs depending on the underlying graph considered. 
Recall that A \emph{graph $G=(V,E)$} is a hypergraph such that for every $e \in E$, $|e|=2$.
Given a hypergraph $H=(V,E)$, we associate two graphs to it. The \emph{primal graph $\prim[H]$ of $H$} is defined as the graph whose vertex set is $V$ and edge set is $\{\{u,v\} \mid u \neq v, \exists e \in E, \ u,v \in e \}$. Intuitively, the primal graph is obtained by replacing every edge of $H$ by a clique. The \emph{incidence graph $\inc[H]$ of $H$} is defined as the bipartite graph whose vertex set is $V \cup E$ and the edge set is $\{\{v,e\} \mid v \in V, e \in E, v \in e\}$. Similarly, the \emph{primal graph $\prim[F]$ of a CNF formula $F$} is the graph whose vertices are the variables of $F$ and such that there is an edge between variables $x$ and $y$ if and only if there is a clause $C \in F$ such that $x, y \in \var{C}$. The \emph{incidence graph $\inc[F]$ of $F$} is the bipartite graph on vertices $\var{F} \cup F$ such that there is an edge between a variable $x$ and a clause $C$ of $F$  if and only if $x \in \var{C}$. 

Given a hypergraph $H$, its \emph{primal treewidth} $\ptw{H}$ is defined as the treewidth of its primal graph, that is $\ptw{H}=\tw{\prim[H]}$, while its \emph{incidence treewidth} $\itw{H}$ is defined as the treewidth of its incidence graph, that is $\itw{H}=\tw{\inc[H]}$.  
The \emph{primal treewidth of $F$} is defined to be $\tw{\prim[F]}$. 
The \emph{incidence treewidth} of $F$ is defined to be $\tw{\inc[F]}$. 

%




We now prove a relation between the incidence treewidth of $F_H$ and of the incidence treewidth of $H$:
\begin{lemma}
\label{lem:itwFh}
    For every $H=(V,E)$, we have  $\itw{F_H} \leq 2\cdot(1+\itw{H})$.
\end{lemma}
\begin{proof}
  The proof is based on the following observation. Let $G'=(V',E')$ be the graph obtained from $\inc[H]$ as follows: we introduce for every edge $e \in E$ a vertex $y_e$ that is connected to $e$ and to every $v \in e$.  Clearly, $e$ and $y_e$ have the same neighborhood in $G'$ and $\inc[H]$ can be obtained by contracting $e$ and $y_e$ for every $e \in E$. Such a transformation is known as a \emph{modules contraction} in graph theory and one can easily prove that the treewidth of $G'$ is at most $2(1+\itw{H})$ (see \cref{lem:twmodule} in \cref{appendix graphs} for a full proof). Now, one can get $\inc[F_H]$ by renaming every vertex $v$ of $G'$ into $x_v$ and every vertex $e$ into $R_e$ and by splitting the edge $\{y_e,x_v\}$ of $G'$ with a node $L_{e,v}$. Hence $\inc[F_H]$ is obtained from $G'$ by splitting edges, an operation that preserves treewidth (see \cref{lem:twsub} in \cref{appendix graphs} for details). In other words, $\itw{F_H} \leq \tw{G'} \leq 2(1+\itw{H})$.
\end{proof}

We observe here that bounded incidence treewidth also captures a class of hypergraphs that is known to be tractable for BPO: the class of cycle hypergraphs~\cite{dPDiG21IJO}. We define a \emph{cycle hypergraph} as a hypergraph $H=(V,E)$ such that $E=\{e_0,\dots,e_{n-1}\}$ and $e_i \cap e_j \neq \emptyset$ if and only if $j = i+1 \mod n$. We remark that this definition is slightly more general than the one in \cite{dPDiG21IJO}. While the primal treewidth of a cycle hypergraph can be arbitrarily large (since an edge of size $k$ induces a $k$-clique in the primal graph and hence treewidth larger than $k$), the incidence treewidth of cycle hypergraphs is $2$.

\begin{lemma}
  For every cycle hypergraph $H=(V,E)$, we have $\itw{H} \leq 2$.
\end{lemma}
\begin{proof}
  We give a tree decomposition $T$ for $\inc[H]$ of width $2$. We start by introducing $n$ bags $B_0, \dots, B_{n-1}$ connected by a path such that $B_0 = \{e_0,e_1\}$ and $B_i = \{e_0, e_i, e_{i+1} \}$ for $1 \leq i < n-1$ and $B_{n-1} = \{e_0, e_{n-1}\}$. 
  We connect to $B_i$ one bag $\{e_i, e_{i+1 \mod n}, v\}$ for every $v \in e_i \cap e_{(i+1) \mod n}$ and one bag $\{e_i, v\}$ for every $v \in e_i \setminus (e_{i-1 \mod n} \cup e_{i+1 \mod n})$. 
  We claim that this tree decomposition, which clearly has width $2$, is a tree decomposition of $\inc[H]$. Indeed, every edge of $\inc[H]$ is of the form $(e_i,v)$. Moreover, since $H$ is a cycle hypergraph:
  \begin{itemize}
  \item Either $v \in e_{i-1 \mod n}$ in which case the edge $(e_i,v)$ is covered by the bag $\{e_{i-1 \mod n}, e_i, v\}$ connected to $B_{i-1}$. 
  \item Either $v \in e_{i+1 \mod n}$ in which case the edge $(e_i,v)$ is covered by the bag $\{e_{i+1 \mod n}, e_i, v\}$ connected to $B_{i}$.
  \item Or $v \in e_i \setminus (e_{i-1 \mod n} \cup e_{i+1 \mod n})$ in which case the edge $(e_i,v)$ is covered by the bag $\{e_i,v\}$ connected to $B_i$.
  \end{itemize}
  Hence every edge of $\inc[H]$ is covered by $T$. 
  It remains to show that the connectedness requirement holds.
  Let $w$ be a vertex of $\inc[H]$. Either $w$ is an edge $e_i$ of $H$ or it is a vertex of $H$. Let us start from examining the case where $w$ is an edge of $H$. 
  If $i = 0$, the subgraph induced by $w$ 
  is connected since $e_0$ appears in every $B_i$ and every other bag is connected to some $B_i$. 
  If $w = e_i$, for   $i>0$, then   $e_i$ appears only in $B_i$, $B_{i+1}$ or in bags connected to $B_i$ or $B_{i+1}$. Hence the subgraph of $T$ induced by $e_i$ is connected. Now $w$ may be a vertex $v \in V$. It is easy to see however that $v$ appears in at most one bag of the tree decomposition and is hence connected.
\end{proof}
 
\subsection{A CNF encoding preserving $\beta$-acyclicity}\label{sec: betaacyclic}

Several generalizations of graph acyclicity have been introduced for hypergraphs~\cite{fagin1983degrees}, but in this paper we will be interested in the notion of $\beta$-acyclicity. A good overview of notions of acyclicities for hypergraphs can be found in~\cite{brault2016hypergraph}. They are particularly interesting since both CNF and BPO are tractable on $\beta$-acyclic hypergraphs. 

Let $H=(V,E)$ be a  hypergraph. 
For a vertex $v \in V$, we denote by $E(v) = \{e \in E \mid v \in e\}$ the set of edges containing $v$.
A \emph{nest point of $H$} is a vertex $v \in V$ such that $E(v)$ is ordered by inclusion, that is, for any $e,f \in E(v)$ either $e \subseteq f$ or $f \subseteq e$. $H$ is $\beta$-acyclic if there exists an ordering $(v_1, \dots, v_n)$ of $V$ such that for every $i \leq n$, $v_i$ is a nest point of $H \setminus \{v_1, \dots, v_{i-1}\}$, where $H \setminus W$ is the hypergraph with vertex set $V \setminus W$ and edge set $\{e \setminus W \mid e \in E\}$. Such ordering is called a \emph{$\beta$-elimination order}. 
A hypergraph is \emph{$\beta$-acyclic} if and only if it has a $\beta$-elimination order.  
For a CNF formula, the \emph{hypergraph $H(F)$ of $F$} is the hypergraph whose vertices are $\var{F}$ and edges are $\var{C}$ for every $C \in F$.\footnote{We observe that the incidence graph of a CNF formula is not the same as the incidence graph of its hypergraph. Indeed, if a CNF formula has two clauses $C,D$ with $\var{C}=\var{D}$, then they will give the same hyperedge in $H$ while they will be distinguished in the incidence graph of $F$. However, one can observe that $C$ and $D$ are \emph{modules} in $\inc[F]$.} A CNF formula is said to be $\beta$-acyclic if and only if $H(F)$ is $\beta$-acyclic.

It has been recently proven that the BPO problem is tractable when its underlying hypergraph is $\beta$-acyclic~\cite{dPDiG22SODA,dPDiG23ALG}. This result is akin to known tractability for $\beta$-acyclic CNF formulas~\cite{Capelli17}. To connect both results, it is tempting to use the CNF encoding $F_H$ from previous section. However, one can construct a $\beta$-acyclic hypergraph $H$ where $F_H$ is not $\beta$-acyclic. This can actually be seen on the example from \cref{fig: multilinear}, which is $\beta$-acyclic since $\{v_1,\dots,v_6\}$ is a $\beta$-elimination order but $F_H$, depicted on \ref{fig: first}, is not $\beta$-acyclic, since it contains, for example, the $\beta$-cycle formed by the vertices $\{x_{v_2}, y_{e_1}, x_{v_3}, y_{e_2} \}$ and edges $\{c_2, c_3, c_9, c_8\}$. 


We can nevertheless design a CNF encoding of $f_H$ for $H=(V,E)$ that preserves $\beta$-acyclicity, as follows: given a total order $\prec$ on $V$, we let $F_H^\prec$ be a CNF on variables $X \cup Y$ with $X=(x_v)_{v \in V}$ and $Y=(y_e)_{e \in E}$ having clauses:
\begin{itemize}
\item for every $e \in E$, there is a clause $R_e=y_e \vee \bigvee_{v \in e} \neg x_v$,
\item for every $e \in E$ and $v \in e$, there is a clause $L^\prec_{v,e} = \neg y_e \vee x_v \vee \bigvee_{w \in e, v \prec w} \neg x_w$.
\end{itemize}

The CNF encoding from \cref{example} following the variable order $(v_1,\dots,v_6)$ and its hypergraph are depicted on \cref{fig: second}. 
It is easy to see that this hypergraph is $\beta$-acyclic.

\begin{figure}[h]
\centering

\begin{minipage}{.4 \textwidth}
\includegraphics[width=6cm]{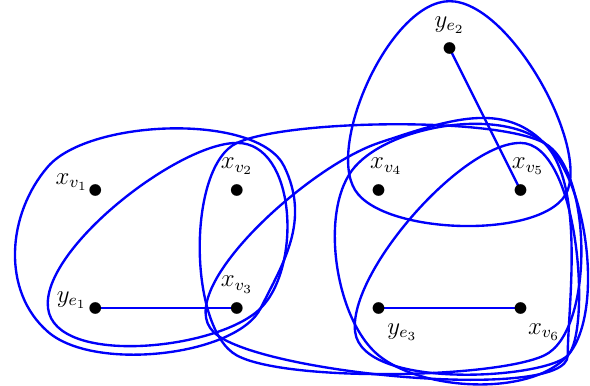}
\end{minipage}
\begin{minipage}{.3\textwidth}
\begingroup
\everymath{\scriptsize}
\scriptsize
\begin{gather*}
  (\neg y_{e_1} \vee x_{v_1} \vee \neg x_{v_2} \vee \neg x_{v_3}) \wedge (\neg y_{e_1} \vee x_{v_2} \vee \neg x_{v_3}) \wedge \\
  (\neg y_{e_1} \vee x_{v_3} ) \wedge (\neg x_{v_1} \vee \neg x_{v_2} \vee \neg x_{v_3} \vee y_{e_1}) \wedge \\ 
  (\neg y_{e_2} \vee x_{v_4} \vee \neg x_{v_5}) \wedge (\neg y_{e_2} \vee x_{v_5}) \wedge 
  (\neg x_{v_4} \vee \neg x_{v_5} \vee y_{e_2}) \wedge \\
  (\neg y_{e_3} \vee x_{v_2} \vee \neg x_{v_3} \vee \neg x_{v_4} \vee \neg x_{v_5} \vee \neg x_{v_6}) \wedge \\
  (\neg y_{e_3} \vee x_{v_3} \vee \neg x_{v_4} \vee \neg x_{v_5} \vee \neg x_{v_6}) \wedge \\
  (\neg y_{e_3} \vee x_{v_4} \vee \neg x_{v_5} \vee \neg x_{v_6}) \wedge (\neg y_{e_3} \vee x_{v_5} \vee \neg x_{v_6}) \wedge \\
  (\neg y_{e_3} \vee x_{v_6} ) \wedge
(\neg x_{v_2} \vee \neg x_{v_3} \vee \neg x_{v_4} \vee \neg x_{v_5} \vee \neg x_{v_6} \vee y_{e_3}).
\end{gather*}
\endgroup
\end{minipage}

\caption{Hypergraph and CNF for $\beta$-acyclic preserving encoding.}\label{fig: second}
\end{figure}

As before, the satisfying assignments of $F_H^\prec$ are the same as $f_H$ since $R_e$ encodes the constraint $\bigwedge_{v \in V}x_v \Rightarrow y_e$ and $\bigwedge_{v \in e} L_{v,e}^\prec$ encodes $y_e \Rightarrow \bigwedge_{v \in e} x_v$.

\begin{lemma}
  \label{lem:FHprec_encodes_fH}
    For every hypergraph $H=(V,E)$ and order $\prec$ on $V$, the satisfying assignments of $F_H^\prec$ are exactly $f_H$.
\end{lemma}

This encoding now preserves $\beta$-acyclicity when the right order is used:
\begin{lemma}
  \label{lem:betapreserved}
  Let $H=(V,E)$ be a $\beta$-acyclic hypergraph and $\prec$ be a $\beta$-elimination order for $H$. Then $F_H^\prec$ is $\beta$-acyclic.
\end{lemma}
\begin{proof}
  We claim that $(y_{e_1}, \dots, y_{e_m}, x_{v_1}, \dots, x_{v_n})$ is a $\beta$-elimination order of $H(F_H^\prec)$, where $V = \{v_1, \dots, v_n\}$ with $v_i \prec v_j$ for any $i<j$ and $E = \{e_1, \dots, e_m\}$ is an arbitrary order on $E$. First we observe that for any $e \in E$, $y_e$ is a nest point of $H(F_H^\prec)$. Indeed, $y_e$ is exactly in $R_e$ and $L_{v,e}$ for every $v \in e$. The fact that $y_e$ is a nest point follows from the fact that $\var{L_{v,e}} \subseteq \var{L_{w,e}}$ if $w \prec v$ and that $\var{L_{v,e}} \subseteq \var{R_e}$ by definition.

  Hence, one can eliminate $y_{e_1}, \dots, y_{e_m}$ from $H(F_H^\prec)$. Let $H'$ be the resulting hypergraph. Observe that the vertices of $H'$ are $\{x_{v_1}, \dots, x_{v_n}\}$. Our goal is to show that $(x_{v_1}, \dots, x_{v_n})$ is a $\beta$-elimination order for $H'$ and we will do it using the fact that $(v_1, \dots, v_n)$ is a $\beta$-elimination order of $H$. 
  In order to ease the comparison between the two hypergraphs, we consider $H'$ to be a hypergraph on vertices $v_1, \dots, v_n$ by simply renaming $x_{v_i}$ to $v_i$. 

  Now, let $V_i = \{v_1, \dots, v_i\}$. We claim that if $f$ is an edge of $H'$ containing $v_{i+1}$ then $f \setminus V_i$ is an edge of $H \setminus V_i$.  
  Indeed, either $f$ is an edge of $H'$ corresponding to a clause $R_e$. In this case, $f = \var{R_e} \setminus \{y_e\}$ since $y_e$ has been removed in $H'$ and hence $f = e$, by definition of $R_e$. Hence $f$ is an edge of $H$, that is, $f \setminus V_i$ is an edge of $H \setminus V_i$.

  Otherwise, $f$ is of the form $\var{L_{v,e}^\prec}\setminus\{y_e\}$ for some $e \in E$ and $v \in e$, that is, $f = e \cap \{w \in V \mid v \preceq w \}$. Since $v_{i+1}$ is in $f$, it means that $v \prec v_{i+1}$ or $v = v_{i+1}$. Hence $f \setminus V_i = e \setminus V_i$.

  Now it is easy to see that $v_{i+1}$ is a nest point of $H' \setminus V_{i}$. Indeed, let $e,f$ be two edges of $H' \setminus V_{i}$ that contains $v_{i+1}$. By definition, $e=e'\setminus V_i$ and $f=f'\setminus V_i$ for some edges $e',f'$ of $H'$. 
  However, 
  $e$ and $f$ are also edges of $H \setminus V_i$. By definition, $v_{i+1}$ is a nest point of $H \setminus V_i$. Hence $e \subseteq f$ or $f \subseteq e$. In any case, it means that $v_{i+1}$ is a nest point of $H' \setminus V_i$ and hence, $(v_1,\dots,v_n)$ is a $\beta$-elimination order of $H'$. 
\end{proof}

\section{Solving BPO using Knowledge Compilation}
\label{sec:bpokc}

The complexity of finding an optimal solution for a \BOP{} $(f,w)$ given a Boolean function $f$ and a weight function $w$ depends on the way $f$ is represented in the input. 
If $f$ is given in the input as a CNF formula, then deciding $w(f)=0$ is $\NP$-hard, since even deciding whether $f$ is satisfiable is $\NP$-hard. Now, if $f$ is given on the input as a truth table, that is, as the list of its satisfying assignments, then it is straightforward to compute $w(f)$ in polynomial time by just computing explicitly $w(\tau)$ for every $\tau \in f$ and return the largest value. This algorithm is polynomial in the size of the truth table, hence, computing $w(f)$ in this setting is tractable. 
However, one can easily observe that this algorithm is polynomial time only because the input was given in a very explicit and non-succinct manner. 
The domain of Knowledge Compilation~\cite{KautzS96} has focused on the study of the representations of Boolean functions, together with their properties, and, more generally, on analysing 
the tradeoff between the succinctness and the tractability of representations of Boolean functions. 
An overview presenting these representations and what is tractable on them can be found in the survey by Darwiche and Marquis~\cite{DarwicheM2002}.

In this paper, we are interested in the \emph{deterministic Decomposable Negation Normal Form} representation, d-DNNF for short, a representation introduced in~\cite{Darwich01}. This data structure has two advantages for our purposes. 
On the one hand, given a Boolean function $f$ represented as a d-DNNF $C$ and a weight function $w$, one can find an optimal solution for $(f,w)$ in polynomial time in the size of $C$~\cite{bourhis20topk,kimmig2017algebraic}. 
On the other hand, d-DNNF can succinctly encode CNF formulas having bounded incidence treewidth or $\beta$-acyclic CNF formulas~\cite{Capelli17,BovaCMS15}.

A \emph{Negation Normal Form circuit} (NNF circuit for short) $C$ on variables $X$ is a labeled directed acyclic graph (DAG) such that:
\begin{itemize}
\item $C$ has a specific node $v$ called the output of $C$ and denoted by $\out$.
\item Every node of $G$ of in-degree $0$ is called an \emph{input of $C$} and is labeled by either $0$, $1$ or a literal $\ell$ over $X$.
\item Every other node of $G$ is labeled by either $\wedge$ or $\vee$. Let $v$ be such a node and  $w$ be such that there is an oriented edge from $w$ to $v$ in $C$. We say that $w$ is \emph{an input of $v$} and we denote by $\inputs$ the set of inputs of $v$.
\end{itemize}
Given a node $v$ of $C$, we denote by $\var{v}$ the subset of $X$ containing the variables $y$ such that there is path from an input labeled by $y$ or $\neg y$ to $v$. Each node of $C$ computes a Boolean function $f_v \subseteq 2^{\var{v}}$ inductively defined as:
\begin{itemize}
\item If $v$ is an input labeled by $0$ (respectively $1$), then $f_v = \emptyset$ (respectively $f_v=2^\emptyset$)
\item If $v$ is an input labeled by $x$ (respectively $\neg x$) then $f_v$ is the Boolean function over $\{x\}$ that contains exactly the one assignment mapping $x$ to $1$ (respectively to $0$)
\item If $v$ is labeled by $\vee$ then $\tau \in f_v$ if and only if there exists $w \in \inputs$ such that $\tau|_{\var{w}} \in f_w$. In other words, $f_v = \bigvee_{w \in \inputs}f_w$.
\item If $v$ is labeled by $\wedge$ then $\tau \in f_v$ if and only if for every $w \in \inputs$ we have $\tau|_{\var{w}} \in f_w$. In other words, $f_v = \bigwedge_{w \in \inputs} f_w$.
\end{itemize}
The Boolean function $f_C$ computed by $C$ over $X$ is defined as $f_{\out} \times 2^{X \setminus \var{\out}}$. 
The \emph{size} of $C$ is defined as the number of edges in the circuit and denoted by $|C|$. 

\begin{figure}
  \centering
  \begin{minipage}{.5\textwidth}
  \includegraphics{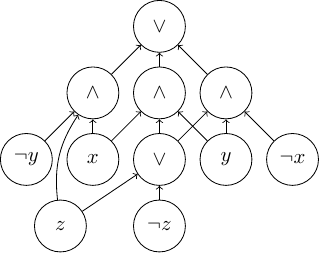}
\end{minipage}
\begin{minipage}{.2\textwidth}
\begin{longtable}[]{@{}lll@{}}
      \toprule
      x & y & z \\
      \midrule
      \endhead
      0 & 1 & 0 \\
      0 & 1 & 1 \\
      1 & 0 & 0 \\
      1 & 1 & 0 \\
      1 & 1 & 1 \\
      \bottomrule
    \end{longtable}
  \end{minipage}
  \caption{A deterministic DNNF and the Boolean function it computes. Observe that each $\wedge$-node mentions disjoint variables ($x$ or $y$ or $z$) and the each $\vee$-node is deterministic: the bottom one because they differ on the value of $z$ and the top one because they differ on the value assigned to $x$ and $y$.}
  \label{fig:ddnnf}
\end{figure}

A $\wedge$-node $v$ is said to be \emph{decomposable} if for every $w,w' \in \inputs$, $\var{w} \cap \var{w'} = \emptyset$. An NNF circuit is a \emph{Decomposable Negation Normal Form circuit}, DNNF for short, if every $\wedge$-node of the circuit is decomposable.

A $\vee$-node $v$ is said to be \emph{deterministic} if for every $\tau \in f_v$, there exists a unique $w \in \inputs$ such that $\tau|_{\var{w}} \in f_w$\footnote{This property is more commonly known as ``unambiguity'' but we stick to the usual terminology of KC in this work.}. A \emph{deterministic Decomposable Negation Normal Form circuit}, d-DNNF for short, is a DNNF where every $\vee$-node is deterministic. An example of d-DNNF is given on \cref{fig:ddnnf}.

Observe that in a d-DNNF circuit, the following identities hold:
\begin{itemize}
\item If $v$ is a $\vee$-node then $f_v = \biguplus_{w \in \inputs} \big(f_w \times 2^{\var{v} \setminus \var{w}}\big)$, where $\uplus$ denotes disjoint union;
\item If $v$ is a $\wedge$-node then $f_v = \bigtimes_{w \in \inputs} f_w$. 
\end{itemize}

These identities allow to design efficient dynamic programming schemes on d-DNNF circuits. In this paper, we will be mostly interested in such dynamic programming scheme to compute optimal solutions to \BOP{}. 
The next theorem is a direct consequence of \cite[Theorem 1]{bourhis20topk} but, since it is stated in a slightly different framework there, we also provide a proof sketch. In this paper, the \emph{size} of numbers, denoted by $\size(\cdot)$, is the standard one in mathematical programming, and is essentially the number of bits required to encode such numbers.
We refer the reader to \cite{SchBookIP} for a formal definition of size.
\begin{theorem}
  \label{thm:maxplus_amc} Let $(f,w)$ be a \BOP{} and $C$ a d-DNNF such that $f = f_C$. We can find an optimal solution $\tau^* \in f$ for $w$ in strongly polynomial time, i.e., with $O(\poly(|C|))$ arithmetic operations on numbers of size $\poly(size(w))$.
\end{theorem}
\begin{proof}[Proof (sketch).]
  The idea is to  compute inductively, for each node $v$ of $C$, an optimal solution $\tau_v^*$ of the \BOP{} problem $(f_v, w)$, that is, a solution $\tau_v^* \in f_v$ such that $w(\tau_v^*)$ is maximal. 
  We recall that $w(\tau_v^*)$ is equal to $\sum_{x \in \var{v}} w(\tau_v^*, \tau_v^*(x))$. 
  If $v$ is a leaf labeled by literal $\ell$, then $f_v$ contains exactly one assignment, which is hence optimal by definition.

  Now let $v$ be an internal node of $C$ and let $v_1, \dots, v_k$ be its input. Assume that for each $v_i$, an optimal solution $\tau_{v_i}^*$ of $(f_{v_i}, w)$ has been precomputed. If $v$ is a $\wedge$-node, then we have that $f_v = \bigtimes_{i=1}^k f_{v_i}$, by the fact that the circuit is decomposable. It can be easily verified that  $\tau^*_v := \bigtimes_{i=1}^k \tau^*_{v_i}$  is an optimal solution of $f_v$ which proves the induction step in this case. 
  Otherwise, $v$ must be a $\vee$-node and in this case, $f_v = \bigcup_{i=1}^k f_{v_i} \times \{0,1\}^{\var{v} \setminus \var{v_i}}$. An optimal solution $\sigma_i^*$ for $\{0,1\}^{\var{v} \setminus \var{v_i}}$ can be obtained by taking $\sigma_i^*(x) = b \in \{0,1\}$ if $w(x,b) \geq w(x,1-b)$ for every $x \in \var{v} \setminus \var{v_i}$. In this case $ \tau^*_{v_i} \times \sigma^*_i$ is an optimal solution of $f_{v_i} \times \{0,1\}^{\var{v} \setminus \var{v_i}}$. Finally, observe that an optimal solution of $f_v$ has to be an optimal solution for $f_{v_i} \times \{0,1\}^{\var{v} \setminus \var{v_i}}$ for some $i$. Hence, $\tau_v^* = \argmax_{i=1}^k w(\tau_{v_i}^* \times \sigma_i^*)$ is an optimal solution of $f_v$.

    At each gate, we have to compare at most $O(|C|)$ numbers, each of size $\poly(\size(w))$, to decide which optimal solution to take. Since we are doing this at most once per gate, we have an algorithm that computes an optimal solution of $f=f_C$ with $\poly(|C|)$ arithmetic operations on number of size at most $\poly(\size(w))$.
\end{proof}

\cref{thm:maxplus_amc} becomes interesting for our purpose since it is known that both $\beta$-acyclic CNF formulas~\cite{Capelli17} and bounded incidence treewidth CNF formulas~\cite{BovaCMS15} can be represented as polynomial size d-DNNF:

\begin{theorem}[\cite{Capelli17}]
  \label{thm:kcbeta} Let $F$ be a $\beta$-acyclic CNF formula having $n$ variables and $m$ clauses. One can construct a d-DNNF $C$ of size $\poly(n,m)$ in time $\poly(n,m)$ such that $f_C$ is the set of satisfying assignments of $F$.
\end{theorem}

\begin{theorem}[\cite{BovaCMS15}]
  \label{thm:kctw} Let $F$ be a CNF formula having $n$ variables, $m$ clauses and incidence treewidth $t$. One can construct a d-DNNF $C$ in time $2^t\poly(n,m)$ of size $2^t\poly(n,m)$ such that $f_C$ is the set of satisfying assignments of $F$. 
\end{theorem}

As a consequence, we can represent the multilinear set of a BPO problem $(H,p)$ with a polynomial size d-DNNF if the incidence treewidth of $H$ is bounded:

\begin{theorem}
  \label{thm:bpo-tw-dnnf} Let $H = (V,E)$ and $t = \itw{H}$. The multilinear set $f_H$ of $H$ can be computed by a d-DNNF of size $2^{O(t)}\poly(|V|,|E|)$ that can be constructed in time $2^{O(t)}\poly(|V|,|E|)$. 
\end{theorem}
\begin{proof}
  We construct $F_H$ from $H$ in time $\poly(|E|,|V|)$. It has $n=|V|+|E|$ variables, $m \leq |E|(1+|V|)$ clauses and incidence treewidth at most $2(t+1)$ by \cref{lem:itwFh}. Hence the construction follows by applying \cref{thm:kctw} to $F_H$.
\end{proof}

Symmetrically, if $H$ is $\beta$-acyclic we get:
\begin{theorem}
  \label{thm:bpo-beta-dnnf} Let $H = (V,E)$ be a $\beta$-acyclic hypergraph. The multilinear set $f_H$ of $H$ can be computed by a d-DNNF of size $\poly(|V|,|E|)$ that can be constructed in time $\poly(|V|,|E|)$. 
\end{theorem}

In particular, the computational complexity result presented in \cref{th BPO short} follows from \cref{thm:maxplus_amc} together with \cref{thm:bpo-beta-dnnf} for the first part and with \cref{thm:bpo-tw-dnnf} for the last part.

\section{Beyond BPO}
\label{sec:beyond}

In this section, we show the versatility of our circuit-based approach by showing how we can use it to solve optimization problems beyond BPO.

\subsection{Adding Cardinality Constraints}
\label{sec:cardinality-constraints}

We start by focusing on solving BPO problems that are combined, for example, with cardinality or modulo constraints, that is, a BPO problem with additional constraints on the value of $\sum_{v \in V} x(v)$. 
In this section, we will be interested in the \emph{extended cardinality constraints} that we define as constraints of the form $\sum_{v \in V} x(v) \in S$, for some $S \subseteq [n]$ where $n=|V|$.

Our approach is based on a transformation of the d-DNNF so that they only accept assignments whose number of variables set to $1$ is constrained. The key result is the following one:

\begin{theorem}
  \label{thm:hamming-constraints} Let $C$ be a d-DNNF on variables $Z$ and let $X\subseteq Z$ with $p=|X|$. One can construct a d-DNNF $C'$ in time $O(\poly(|C|))$ and of size $O(\poly(|C|))$ such that there exists nodes $r_0, \dots, r_p$ in $C'$ with $f_{r_i} = f_C \cap B_i(X,Z)$ where $B_i(X,Z)=\{\tau \in \{0,1\}^Z \mid \sum_{x \in X \cap Z} \tau(x)=i\}$ for $i \in [p]$.
\end{theorem}

\begin{proof}
  Recall that each node $v$ of $C$ computes a Boolean function $f_v$ on $\var{v}$ inductively defined in \cref{sec:preliminary}. The main idea of the proof is to construct $C'$ by introducing $p+1$ copies $v^{(0)}, \dots, v^{(p)}$ for every node $v$ of $C$ and plug them together so that $f_{v^{(i)}}$ computes $f_v \cap B_i(X,\var{v})$. 
  Observe that if $|\var{v} \cap X| < p$ then for $|\var{v} \cap X| < j \leq p$, $f_v \cap B_j(X,\var{v}) = \emptyset$, so we do not really need to introduce a copy $v^{(j)}$. However, in order to minimize the number of  cases to consider in the induction, we  introduce $p$ copies for each $v$, regardless of $|\var{v} \cap X|$.

  To make the transformation easier to perform, we first assume that $C$ has been normalized so that for every $\vee$-node $v$ and input $w$ of $v$, we have $\var{v}=\var{w}$ (a property known as \emph{smoothness}~\cite{DarwicheM2002}) and such that every $\wedge$-node has at most two inputs. It is known that one can normalize $C$ in polynomial time (see \cite[Theorem 1.61]{CapelliPhD} for details).
  
  We now construct from $C$ a new circuit $C'$ of size at most $3p^2|C|$ that will have the following property: for each node $v$ of $C$, there exists $v^{(0)}, \dots, v^{(p)}$ in $C'$ such that $v^{(i)}$ computes $f_v \cap B_i(X,\var{v})$.  The construction of $C'$ is done by induction on the size of $C$.

  We start with the base case where $C$ contains exactly one node, which is then necessarely an input $v$. If $v$ is labeled by $0$, $1$, $y$ or $\neg y$ with $y \notin X$ or by $\neg x$ with $x \in X$ then we add an input node $v^{(0)}$ with the same label as $v$ and introduce 
  $0$-input $v^{(1)}, \dots, v^{(p)}$. 
  If $v$ is an input node of $C$ labeled by $x \in X$, we introduce an input labeled by $x$ as $v^{(1)}$ and let $v^{(0)}, v^{(2)}, \dots, v^{(p)}$ be $0$-inputs. 
  It can readily be verified that for every previous case, that is, for every case where $v$ is an input, then $f_{v} \cap B_i(X,\var{v^{(i)}})$ is the same as $f_{v^{(i)}}$ for every $i \in [n]$. Moreover, $|C'|\leq p = 3p^2|C|$ in this case, which respects the induction hypothesis.

  We now proceed to the inductive step. Let $v$ be node of $C$ that has outdegree $0$ and let $C_0$ be the circuit obtained by removing $v$ from $C$. By induction, we can construct a circuit $C'_0$ of size at most $3p^2|C_0|=3p^2|C|-3p^2$, such that for every node
  $w$ of $C_0$ -- that is for every node of $C$ that is different from $v$, there exists $w^{(0)}, \dots, w^{(p)}$ in $C_0'$ such that $w^{(i)}$ computes $f_w \cap B_i(X,\var{w})$. Observe that $w$ computes the same function in $C_0$ or in $C$ since we have removed a gate of outdegree $0$ that is hence not used in the computation of any other $f_w$. Hence, to construct $C'$, it is enough to add new vertices  $v^{(0)}, \dots, v^{(p)}$ in $C_0'$ such that $v^{(i)}$ computes $f_v \cap B_i(X, \var{v})$. We now explain how to do it by adding at most $3p^2$ nodes in $C_0'$, resulting in a circuit $C'$ that will respect the induction hypothesis. 

  We start with the case where $v$ is a $\vee$-node. Let $w_1, \dots, w_k$ be the inputs of $v$. For every $i \in [p]$, we create a new $\vee$-node $v^{(i)}$  and add it to $C_0'$  and connect it to the nodes $w_j^{(i)}$ of $C_0'$ for every $j \leq k$, where $w_j^{(i)}$ is the node of $C_0'$ computing $f_{w_j} \cap B_i(X,\var{w_j})$ given by induction. 
  
  We claim that $v^{(i)}$ computes $f_v \cap B_i(X,\var{v})$. Indeed, since $C$ is smooth, $\var{w_j} = \var{v}$. Hence we have: $f_{v^{(i)}} = \bigcup_{j=1}^k f_{w_j^{(i)}} = \bigcup_{j=1}^k (f_{w_j} \cap B_i(X,\var{v})) = f_v \cap B_i(X,\var{v})$, which is the induction hypothesis. Moreover $C'$ is  obtained from $C_0'$ by adding at most $p$ new gates, hence $|C'|=|C_0'|+p \leq 3p^2(|C|-1)+p \leq 3p^2|C|$.

  We now deal with the case where  $v$ is a $\wedge$-node of $C$ with inputs $w_1, w_2$ (recall that $C$ has been normalized so that $\wedge$-nodes have two inputs). We construct $C'$ by adding, for every $i \leq p$, a new $\vee$-node $v^{(i)}$ in $C_0'$. Moreover, for every $(a,b) \in [p]^2$ such that $a+b=i$, we add a new node $\wedge$-node $v_{a,b}$.
  
  Let $w_1^{(a)}$ and $w_2^{(b)}$ be the nodes in $C_0'$ computing $f_{w_1} \cap B_a(X,\var{w})$ and $f_{w_2} \cap B_b(X,\var{w})$ respectively. We add the edge  $w_1^{(a)} \rightarrow v_{a,b}$ and $w_2^{(b)} \rightarrow v_{a,b}$ in $C_0'$, that is, $v_{a,b}$ has now $w_1^{(a)}$ and $w_2^{(b)}$ as inputs. Moreover, for every $a,b$ such that $a+b=i$, we add the edge $v_{a,b} \rightarrow v^{(i)}$, that is, the input of $v^{(i)}$ are the set of nodes $v_{a,b}$ such that $a+b=i$. We claim that $v^{(i)}$ computes $f_v \cap B_i(X,\var{v})$.
  
  It is readily verified that $f_{v^{(i)}} =\bigcup_{(a,b) \mid a+b=i} f_{w_1^{(a)}} \times f_{w_2^{(b)}}$. 
  By induction, this is equal to $\bigcup_{(a,b) \mid a+b=i} (f_{w_1} \cap B_a(X,\var{w_1})) \times (f_{w_2} \cap B_b(X,\var{w_2})$. 
  Now, it is easy to see that this set is included in $(f_{w_1} \times f_{w_2}) \cap B_i(X,\var{w}) = f_v \cap B_i(X,\var{v})$. To see the other inclusion, consider $\tau \in f_v \cap B_i(X,\var{v})$. By definition, $\tau = \tau_1 \times \tau_2$ with $\tau_1 \in f_{w_1}$ and $\tau_2 \in f_{w_2}$. 
  Moreover, $i = \sum_{x \in X \cap \var{v}} \tau(x) = \sum_{x \in X \cap \var{w_1}} \tau(x) + \sum_{x \in X \cap \var{w_2}} \tau(x)$. Let $a' = \sum_{x \in X \cap \var{w_1}} \tau(x)$ and $b'=\sum_{x \in X \cap \var{w_2}} \tau(x)$. We have $a'+b'=i$ and $\tau_1=\tau|_{\var{w_1}} \in B_{a'}(X,\var{w_1})$ and $\tau_2 = \tau|_{\var{w_2}} \in B_{b'}(X,\var{w_2})$. Hence $\tau = \tau_1\times\tau_2 \in  (f_{w_1} \cap B_{a'}(X,\var{w_1})) \times (f_{w_2} \cap B_{b'}(X,\var{w_2})$ with $a'+b'=i$. This is enough to conclude that $$\bigcup_{(a,b) \mid a+b=i} (f_{w_1} \cap B_a(X,\var{w_1})) \times (f_{w_2} \cap B_b(X,\var{w_2}) = f_v \cap B_i(X,\var{v}).$$ 

  Finally, observe that $C'$ is obtained from $C_0'$ by adding $3$ edges for each $v_{a,b}$ node, that is, at most $3p^2$ edges. Hence $|C'| \leq |C_0'|+3p^2 \leq 3p^2(|C|-1)+3p^2 = 3p^2|C|$, which concludes the induction. 
  
  Now let $C'$ be the circuit obtained from $C$ by doing the previously describe transformation. Let $r = \out[C]$.  Then by construction, there exists  $r^{(i)}$ in $C'$ computing $f_C \cap B_i(X,Z)$ and $C'$ is of size at most $3p^2|C| = \poly(p)|C|$. Moreover, it is straightforward to see that this construction can be down in polynomial time by induction on the circuit which concludes the proof.
\end{proof}

We will be mostly interested in the following corollary:
\begin{corollary}
  \label{cor:hamming-constraints} Let $S \subseteq [p]$ and $C$ be a d-DNNF on variables $Z$ and $X \subseteq Z$ with $|X|=p$. We can construct a d-DNNF $C'$ of size $\poly(|C|)$ in time $\poly(|C|)$ such that $f_{C'}=f_C \cap B_S(X,Z)$ where $B_S(X,Z) = \bigcup_{i \in S} B_i(X,Z)$. 
\end{corollary}
\begin{proof}
  Let $C_0$ be the circuit given by \cref{thm:hamming-constraints} and define $C'$ by adding a new $\vee$-node $r$ that is connected to $r^{(i)}$ for every $i \in S$ in $C_0$. It is readily verified that $f_r = f_C \cap B_S(X,Z)$. 
\end{proof}

Now \cref{cor:hamming-constraints} allows to solve any \BOP{} $(f,w)$ with an additional constraint $\sum_{x \in X} \tau(x) \in S$ for $S \subseteq [p]$ and $X \subseteq Z$ in polynomial time if $f$ is given as a d-DNNF and has variable set $Z$. Indeed, this optimization problem can be reformulated as the \BOP{} problem  $(f \cap B_S(X,Z), w)$. By \cref{cor:hamming-constraints}, $f \cap B_S(X,Z)$ can be represented by a d-DNNF of size $\poly(|C|)$ that we can use to compute $w(f \cap B_S(X))$ in $\poly(|C|)$ by \cref{thm:maxplus_amc}. 
Hence, the previous discussion and \cref{thm:bpo-tw-dnnf,thm:bpo-beta-dnnf} imply:
\begin{theorem}
  Let $H=(V,E)$ be a hypergraph, $p$ a profit function and $S \subseteq V$. We can solve the optimization problem
  \begin{align*}
  \max_{x \in \{0,1\}^V} & P_{(H,p)}(x) \\ \text{s.t. } & \sum_{x \in V} x(v) \in S
  \end{align*}
in time $\poly(|V|,|E|)$ if $H$ is $\beta$-acyclic and in time $2^{O(t)}\poly(|V|,|E|)$ where $t$ is the incidence treewidth of $H$.
\end{theorem}
\begin{proof}
  We explain it for the $\beta$-acyclic case, the other case being completely symmetric. By \cref{thm:bpo-beta-dnnf}, there is a d-DNNF $C$ computing $f_H$ of size $\poly(|V|,|E|)$. Recall that $f_H$ is defined on variables $X = (x_v)_{v \in V}$ and $Y = (y_e)_{e \in E}$. By \cref{cor:hamming-constraints}, there is a d-DNNF $C'$ computing $f_H \cap B_S(X,X \cup Y)$, that is, the set of $\tau \in f_H$ such that $\sum_{v \in V} \tau(x_v) \in S$. By~\cref{lem:bpotobop}, $f_H \cap B_S(X,X \cup Y)$  is isomorphic to $\mathcal{B}=\{\sigma \in \{0,1\}^V \mid \sum_{v\in V} \sigma(v) \in S\}$ and this isomorphism preserves the weights. Hence $w_p(f_H \cap B_S(X,X \cup Y))$ is the same as $\max_{\sigma \in \mathcal{B}} P_{(H,p)}(\sigma)$, which is the value we are looking for.  Hence we can compute the desired optimal value in $\poly(|C'|)$, that is, $\poly(|V|,|E|)$.
\end{proof}

In particular, BPO problems with cardinality or modulo constraints are tractable on $\beta$-acyclic hypergraphs and bounded incidence treewidth hypergraphs since a conjunction of such constraints can easily be encoded as a single extended cardinality constraints. We insist here on the fact that the modulo and cardinality constraint have to be over every variable of the polynomial, or the structure of the resulting hypergraph would be affected. 

Finally, we observe that the construction from \cref{thm:hamming-constraints} can be generalized to solve \BOP{} problem with one knapsack constraint. A knapsack constraint is a constraint of the form $L \leq \sum_{v \in V} c_v x(v) \leq U$ where $L, U, c_v \in \Z$. The construction is more expensive than the previous one as it will provide an algorithm whose complexity is polynomial in $M = |V| \cdot \max_{v \in V}|c_v|$, which is not polynomial in the input size if the integer values $c_v$ are assumed to be binary encoded. Indeed, we know that for any $x \in \{0,1\}^V$, $\sum_{v \in V} c_v x(v)$ will be between $-M$ and $M$. Hence, we can modify the construction of \cref{thm:hamming-constraints} so that $C'$ has a node $r^{(i)}$ for every $i \in [-M;M]$ and such that $f_{r^{(i)}} = f_C \cap \{\tau \in \{0,1\}^V \mid \sum c_vx(v) = i \}$. 

\subsection{Top-$k$ for BPO}
\label{sec:top-k-bpo}

We now turn our attention to the problem of computing $k$ best solutions for a given $k$. 
This task is natural when there is more than one optimal solution and one wants to explore them in order to find a solution that is more suitable to the user's needs than the 
one found by the solver. Formally, a top-$k$ set for a BPO instance $(H,p)$ is a set $K = \{\sigma_1, \dots, \sigma_k\} \subseteq \{0,1\}^V$ of size $k$ such that $P_{(H,p)}(\sigma_1) \geq \dots \geq P_{(H,p)}(\sigma_k)$ and such that for every $\sigma \notin K$, $P_{(H,p)}(\sigma) \leq P_{(H,p)}(\sigma_k)$. 
We observe that a top-$k$ set for a BPO instance may not be unique, as two distinct $\sigma, \sigma'$ may have the same value under $P_{(H,p)}$. The top-$k$ problem asks for outputting a top-$k$ set given a BPO instance $(H,p)$ and $k \in \mathbb{N}$ on input. 

We similarly define the top-$k$ problem for a \BOP{} $(f,w)$. It turns out that finding a top-$k$ set for $f$ when $f$ is given as a d-DNNF $C$ is tractable. The proof is very similar to the one of \cref{thm:maxplus_amc} but instead of constructing an optimal solution $\tau_v^*$ of $f_v$ for each gate $v$, we construct a top-$k$ set $K_v$ for $f_v$. We can build such top-$k$ sets in a bottom up fashion by observing that if the circuit has been normalized as in \cref{thm:hamming-constraints}, we have that if $v$ is a $\vee$-gate with input $v_1,\dots,v_p$, then a top-$k$ set of $f_v$ can be found by taking a top-$k$ set of $\bigcup_{i=1}^p K_{v_i}$, where $K_{v_i}$ is an top-$k$ set of $f_{v_i}$ constructed inductively. Similarly, if $v$ is a $\wedge$-gate with input $v_1,v_2$, a top-$k$ set for $f_v$ can be found by taking a top-$k$ set of $K_{v_1} \times K_{v_2}$. The following summarizes the above discussion and can be found in~\cite{bourhis20topk}:

\begin{theorem}[Reformulated from \cite{bourhis20topk}]
\label{thm:topkdnnf}Given a d-DNNF $C$ on variables $Z$ and a weight function $w$ on $Z$, one can compute a top-$k$ set for the \BOP{} $(f_C,w)$ in time $O(|C| \cdot k \log k)$.
\end{theorem}

Using the weight preserving isomorphism from \cref{lem:bpotobop}, it is clear that if $K=\{\tau_1, \dots, \tau_k\}$ is a top-$k$ set for $(f_H,w_p)$, then $L=\{x_{\tau_1}, \dots, x_{\tau_k}\}$ is a top-$k$ set for $\BPO$. Hence, by \cref{thm:kctw,thm:kcbeta} together with \cref{thm:topkdnnf}:

\begin{theorem}
  Let $H=(V,E)$ be a hypergraph and $p$ a profit function.  We can compute a top-$k$ set for the BPO instance $(H,p)$ in time $\poly(k,|V|,|E|)$ if $H$ is $\beta$-acyclic and in time $2^{O(t)}\poly(k,|V|,|E|)$ if $H$ has incidence treewidth $t$.
\end{theorem}

\subsection{Solving Binary Polynomial Optimization with Literals}\label{sec:pseudo-bool}


The relation between BPO and \BOP{} from \cref{sec:bpokc} can naturally be extended to Binary Polynomial Optimization with Literals, abbreviated \BPOL{}. In \BPOL{}, the goal is to maximize a polynomial over binary values given as a sum of monomials over literals, that is, each monomial is of the form $p(e) \prod_{v \in e} \sigma_e(v)$ where $\sigma_e(v) \in \{x(v), 1-x(v)\}$. An instance of \BPOL{} is completely characterized by $(H,p,\sigma)$ where $H=(V,E)$ is a multihypergraph --- that is, a hypergraph where $E$ is a multiset, $p \from E \to \Q$ is a profit function, and $\sigma=(\sigma_e)_{e\in E}$ is a family of mappings for $e \in E$ that associates $v \in e$ to an element in $\{x(v), 1-x(v)\}$. We denote by $P_\calH$ the polynomial $\sum_{e \in E} p(e)\prod_{v \in e} \sigma_e(v)$.

Given an instance of \BPOL{} $\calH=(H,\sigma,p)$, we define the Boolean function $f_\calH$ on variables $X=(x_v)_{v \in V}$ and $Y=(y_e)_{e \in E}$ as follows: $\tau \in \{0,1\}^{X \cup Y}$ is in $f_\calH$ if and only if $\tau(y_e) = \prod_{v \in e, \sigma_e(v)=x(v)} \tau(x_v)\prod_{v \in e, \sigma_e(v)=1-x(v)} 1-\tau(x_v)$. As before, we can naturally associate a \BOP{} instance to $\calH$. Indeed, we define $w_\calH$ as before to be the weighting function defined as $w_\calH(x_v) = 0$ for every $v \in V$ and $w_\calH(y_e)=p(e)$ for every $e \in E$. Again, there is a weight preserving isomorphism between $f_\calH$ and $\{0,1\}^V$, mirroring \cref{lem:bpotobop} and \cref{thm:bpotobop}:

\begin{lemma}
  \label{lem:pB:bpotobop}
    For every $\tau \in f_\calH$, we have $P_{\calH}(x_\tau) = w_p(\tau)$ where the mapping $x_\tau \in \{0,1\}^V$ is defined as $x_\tau(v) = \tau(x_v)$ for every $v \in V$.
\end{lemma}

\begin{theorem}
  \label{thm:pB:bpotobop}
  For every instance $\calH=(H,\sigma,p)$ of \BPOL{}, the set of optimal solutions of $\calH$ is equal to $\{x_{\tau^*} \mid \tau^* \text{ is an optimal solution of $(f_\calH,w_p)$}\}$.
\end{theorem}

Now we use this connection and a CNF encoding of $f_\calH$ to compute the optimal value of structured instances of \BPOL{}. As before, $f_\calH$ is equivalently defined as the following Boolean formula: $\bigwedge_{e \in E} y_e \Leftrightarrow \big (\bigwedge_{v \in e, \sigma_e(v)=x(v)} x_v \wedge \bigwedge_{v \in e, \sigma_e(v)=1-x(v)} \neg x_v\big )$ which can be rewritten as $\bigwedge_{e \in E} y_e \Leftrightarrow \bigwedge_{v \in e} \sigma_e(x_v)$ where we abuse the notation $\sigma_e$ by defining $\sigma_e(x_v) = x_v$ if $\sigma_e(v) = x(v)$ and $\neg x_v$ otherwise. Again, $F_\calH$ can be encoded as the conjunction of:
\begin{itemize}
\item $R'_e = y_e \vee \bigvee_{v \in e} \neg \sigma_e(x_v)$ for every $e \in E$ (where we replace $\neg \neg x$ by $x$),
\item $L'_{e,v} = \neg y_e \vee \sigma_e(x_v)$ for every $e \in E$ and $v \in e$.
\end{itemize}

Given a multihypergraph $H=(V,E)$, we define its incidence graph as for hypergraph but we have one vertex for each occurence of $e \in E$ in the multiset $E$. The incidence treewidth $\itw{H}$ of $H$ is defined to be the incidence treewidth of the incidence graph of $H$. The proof of \cref{lem:itwFh} also works for multihypergraphs:

\begin{lemma}
  \label{lem:pB:itwFh} For every instance $\calH=(H,\sigma,p)$ of \BPOL{}, we have that the satisfying assignment of $F_\calH$ are exactly $f_\calH$. Moreover,  $\itw{F_\calH} \leq 2\cdot(1+\itw{H})$.
\end{lemma}

Similarly as before, we need an alternate encoding to handle $\beta$-acyclic instances. If $\prec$ is an order on $V$, we define $F_\calH^\prec$ as the conjunction of:
\begin{itemize}
\item $R'_e = y_e \vee \bigvee_{v \in e} \neg \sigma_e(x_v)$ for every $e \in E$ (where we replace $\neg \neg x$ by $x$),
\item $L'_{e,v,\prec} = \neg y_e \vee \sigma_e(x_v) \vee \bigvee_{w \in e, v \prec w} \neg \sigma_e(x_w)$ for every $e \in E$ and $v \in e$ (where we replace $\neg \neg x$ by $x$).
\end{itemize}

A multihypergraph $H=(V,E)$ is said to be $\beta$-acyclic if it admits a $\beta$-elimination order. One can actually check that this is equivalent to the fact that the hypergraph $H'$ obtained by simply transforming $E$ into a set is $\beta$-acyclic. As before, we have:

\begin{lemma}
  \label{lem:pB:betapreserved}
  Let $\calH=(H,\sigma,p)$ be an instance of \BPOL{} and $\prec$ be a $\beta$-elimination order for $H$. The satisfying assignments of $F_\calH^\prec$ are exactly $f_\calH$ and $F_\calH^\prec$ is $\beta$-acyclic.
\end{lemma}

Hence \cref{lem:pB:itwFh,lem:pB:betapreserved} combined with \cref{thm:kcbeta,thm:kctw} allows to show that:

\begin{theorem}
  \label{thm:pB:compileBPO} Let $\calH=(H,\sigma,p)$ be an instance of \BPOL{}. If $H$ is $\beta$-acyclic (resp. has incidence treewidth $t$), one can construct a d-DNNF computing $f_\calH$ in time $\poly(|V|,|E|)$ (resp. $2^{O(t)}\poly(|V|,|E|)$) of size $\poly(|V|,|E|)$ (resp. $2^{O(t)}\poly(|V|,|E|)$).
\end{theorem}

\cref{thm:pB:compileBPO} with the isomorphism from \cref{lem:pB:bpotobop} allows to show the tractability of \BPOL{} on $\beta$-acyclic and bounded incidence treewidth instances. We can actually go one step further and incorporate the techniques from \cref{sec:cardinality-constraints,sec:top-k-bpo} to get the following tractability result, which is a formalized version of the computational complexity result presented in \cref{th general}:

\begin{theorem}
  \label{thm:everything} Let $\calH=(H,\sigma,p)$ with $H=(V,E)$ be an instance of \BPOL{}, $k \in \mathbb{N}$ and $S \subseteq [n]$ where $n=|V|$. We can solve the following optimization problem $(OPT)$:
  \begin{align*}
    \topk_{x \in \{0,1\}^V} P_\calH(x) \\
    \sum_{v \in V} x(v) \in S
  \end{align*}
  in time $\poly(k,|V|,|E|)$ (resp. $2^{O(t)}\poly(k,|V|,|E|)$) if $H$ is $\beta$-acyclic (resp. $H$ has incidence treewidth $t$).
\end{theorem}
\begin{proof}
  We give the proof for the $\beta$-acyclic case, the bounded incidence treewidth case is completely symmetrical. We start by constructing a d-DNNF computing $f_\calH$ using \cref{thm:pB:compileBPO} and transform it so that it computes $f'=f_\calH \cap B_S(X, X \cup Y)$ using \cref{cor:hamming-constraints}. We use \cref{thm:topkdnnf} to compute a top-$k$ set $K=\{\tau_1,\dots,\tau_k\}$ for $(f',w_\calH)$ and we return $L = \{x_{\tau_1}, \dots, x_{\tau_k}\}$ where $x_\tau$ is the isomorphism from \cref{lem:pB:bpotobop}. The whole computation is in time $\poly(k,|V|,|E|)$.

  It remains to prove that $L$ is indeed a solution of the optimization problem from the statement. We claim that $x_{\tau_i}$ is a feasible point of $(OPT)$. By \cref{lem:pB:bpotobop}, $\tau_i(x_v)=x_{\tau_i}(v)$, hence $\sum_{v \in V} x_{\tau_i}(v) = \sum_{v \in V} \tau_i(x_v) \in S$ since $\tau_i \in B_S(X,X \cup Y)$.

  Now, we prove that $L$ is a top-$k$ set for $(OPT)$. By \cref{lem:pB:bpotobop}, we have $w_\calH(\tau_i) = P_\calH(x_{\tau_i})$. Hence $P_\calH(x_{\tau_1}) \geq \dots \geq P_\calH(x_{\tau_k})$. Moreover, for any feasible point $x \notin L$ of $(OPT)$, we let $\tau  \in f_\calH$ be the unique assignment of $f_\calH$ such that $x = x_\tau$. Observe that since $x$ verifies $\sum_{v \in V} \tau(x_v) = \sum_{v \in V} x(v) \in S$, we have $\tau \in f_H \cap B_S(X,X \cup Y)$. Moreover, $\tau \notin K$ since  $x \notin L$. Hence we have that $w_\calH(\tau) \leq w_\calH(\tau_k)$ since $K$ is a top-$k$ set for $f_H \cap B_S(X,X \cup Y)$. By \cref{lem:pB:bpotobop} again, $w_\calH(\tau) = P_\calH(x) \leq w_\calH(\tau_k) = P_\calH(x_k)$, which establishes that $L$ is a top-$k$ set for $(OPT)$.
\end{proof}

\section{Extended formulations from DNNF}
\label{sec:extend-form}




Another way of efficiently solving BPO is by providing an extended formulation of polynomial size of its multilinear polytope: one can then solve the optimization problem efficiently with linear programming \cite{Kha79,Tar86}. 
Interestingly, for most known tractable classes of BPO, polynomial size extended formulations are known~\cite{WaiJor04,Lau09,BieMun18,dPKha18SIOPT,dPDiG21IJO,dPKha23MPA}. 
In this section, we give a unifying and generalizing view of these results by proving that if the multilinear set $f_H$ of $H$ is represented as a (smooth) DNNF $C$, then there is an extended formulation of $\conv(f_H)$ of size $O(|C|)$.
In particular, this result provides extended formulations of polynomial size corresponding to all the classes of problems described in \cref{sec:beyond}.

Given a set of vectors $P \subseteq \R^V$, the \emph{convex hull of $P$}, denoted by $\conv(P)$ is defined as $\{ \sum_{p\in P} \alpha_p p \mid \forall p \in P, \alpha_p \geq 0 \text{ and } \sum_{p \in P} \alpha_p = 1 \}$. If $C$ is a DNNF, we define the \emph{convex hull of $C$} to be $\conv(f_C^{-1}(1))$ which we will denote by $\conv(C)$ for short. We prove the following:
\begin{theorem}
\label{thm:extended}
For every smooth DNNF $C$, there exists a matrix $A_C, G_C$ such that $\conv(C) = \{x \in \R^V : \exists y \in \R^W \st A_C x+G_C y \leq b\}$. Moreover, $A_C, G_C$ and $b$ have $O(|C|)$ rows and their coefficients are in $\{-1,0,1\}$ and a sparse representation can be computed in $O(|C|)$ given $C$.
\end{theorem}

Extended formulations have historically been known to follow from dynamic programming algorithms. In a way, our result is similar since DNNF can be seen as some trace of a dynamic programming algorithm. Actually, Martin, Rardin and Campbell gave a very general framework in~\cite{MarRarCam90} that could be used on DNNF directly but the translation between both formalism is more tedious than providing an independant proof. 

In this section, we hence fix DNNF $C$ and normalize it as follows, wlog: we first assume that $C$ is smooth since every DNNF $C'$ can be smoothed into a DNNF of size at most $|C| \times |\var{C}|$. Remember that a DNNF is smooth whenever for every $\vee$-gate $g$ and $g'$ a child of $g$, we have $\var{g}=\var{g'}$. We also assume that the output gate of $C$, denoted by $o$, is a $\vee$-gate with no outgoing edges and that there is a path from every gate $g$ of $C$ to $o$. This can be ensured by iteratively removing every gate which does not satisfy this property. Finally, wlog, we assume that for every literal $\ell$, there is at most one input of $C$ labeled by $\ell$. This can be easily ensured by merging several edges labeled by $\ell$. We denote by $\edges[C][\ell]$ the edges going out of the only input labeled by $\ell$. We associate a system of linear inequality constraints to every NNF circuit $C$ as follows. We will use the following variables:
\begin{itemize}
    \item A variable $x_v$ for each $v \in V$,
    \item A variable $y_e$ for each $e \in \edges[C]$.
    \end{itemize}

    In the rest of this section, we use the following notations: for a gate $g$ of $C$, we denote by $\ine[g]$ the set of edges going in $g$ and by $\oute[g]$ the set of edges going out of $g$. We simlarly denote by $\in[g]$ the set of inputs of $g$, that is, the gates $h$ such that $(h,g)$ is an edge of $C$ and by $\out[g]$ the set of outputs of $g$, that is, the gates $h$ such that $(g,h)$ is an edge of $C$. 
    
We define $\S_y(C)$ as the following system of linear constraints in $y$ variables:
\begin{subequations}
\label{eq TDI system}
\begin{align}
    \sum_{e \in \ine[o]} y_e  & = 1,  \label{eq TDI system-o} \\
    \sum_{e \in \ine[g]} y_e - \sum_{f \in \oute[g]} y_f & = 0  & \text{for every $\vee$-gate $g \neq o$,} \label{eq TDI system-or} \\
    y_e - \sum_{f \in \oute[g]} y_f & =0  & \text{for every $\wedge$-gate $g$ and $e \in \ine[g]$} \label{eq TDI system-and} \\
    y_e & \ge 0  & \text{for every $e \in \edges$}. \label{eq TDI system-bounds}
\end{align}
\end{subequations}

We then define $\P_y(C)$ as the polyhedron:
$$
\P_y(C) := \{y \in \R^{n_y} : \text{$y$ satisfies $\S_y(C)$}\},
$$

Now let $\S_{x,y}(C)$ be the system of linear constraints in $x,y$ variables obtained by augmenting $\S_y(C)$ with the following constraints linking $x$ and $y$ variables:
\begin{align}
\label{eq proj}
x_v - \sum_{e \in \edges[C][v]} y_e & = 0 & \text{ for every } v \in V ,
\end{align}
Let $\P_{x,y}(C)$ be the polyhedron:
$$
\P_{x,y}(C) = \{(x,y) \in \R^{n_x+n_y} : \text{$(x,y)$ satisfies $\S_{x,y}(C)$}\}.
$$

\cref{thm:extended} directly follows from \cref{lem:dnnf-polyhedron} which establishes a connection between the projection of integral solutions of $\P_{x,y}(C)$ onto $x$ variables and $f_C^{-1}(1)$ and \cref{lem:pxy-integral} which shows that $\P_{x,y}(C)$ is integral. 

We start by connection $\P_{x,y}(C)$ and $f_C^{-1}(1)$ together:
\begin{lemma}
  \label{lem:dnnf-polyhedron}
  The projection of $\P_{x,y}(C) \cap \Z^{n_x + n_y}$ onto $(x_v)_{v \in V}$ coincides with $f_C^{-1}(1)$.
\end{lemma}

\paragraph{Certificates.}  The proof of \Cref{lem:dnnf-polyhedron} is based on the notion of certificates. Intuitively, we prove that the solution of $\P_{x,y}(C)$ can be seen as some kind of fractional relaxation of the notion of certificate. A certificate is a witness in the circuit that a Boolean assignment is a solution. Formally, a \emph{certificate of $C$} is a subset $T$ of gates of $C$ such that the output of $C$ is in $T$ and such that:
\begin{itemize}
\item for every $\vee$-gate $g$ of $T$, exactly one input of $g$ is in $T$,
\item for every $\wedge$-gate of $g$ of $T$, every input of $g$ is in $T$,
\item for every $g$ in $T$ that is not the output of $C$, $g$ is the input of at least one gate in $T$.
\end{itemize}
We denote by $\cert(C)$ the set of certificate of $C$. In DNNF, it is not hard to see that certificates are trees directed toward the output of the circuit, that is, for each gate $g$ in $T$, there is at least one $g'$ in $T$ such that $(g,g')$ is an edge. 
Moreover, they are related to solutions in $C$ as follows (see \cite[Section 6.1.1]{CapelliPhD} for details). If $C$ is a smooth DNNF, one can easily check that for every certificate $T$ and every variable $v \in V$, there is exactly one input gate labeled with a literal involving variable $v$ (that is, either there is exactly one input labeled $v$ or exactly one input labeled $\neg v$). For a certificate $T$, we denote by $\tau_T$ the assignment defined as $\tau_T(v) = 1$ if an input labeled by $v$ is in $T$ and $\tau_T(v) = 0$ if an input labeled by $\neg v$ is in $T$. We have that $\{\tau_T \mid T \in \cert(C)\}$ is the set of solutions of $C$. 

\begin{proof}[Proof of \cref{lem:dnnf-polyhedron}.] We start by proving that the projection of $\P_{x,y}(C) \cap \Z^{n_x + n_y}$ onto $(x_v)_{v \in V}$ coincides with $f_C^{-1}(1)$. To do so, we will show that $\P_{x,y}(C) \cap \Z^{n_x + n_y}$ naturally corresponds to $\cert(C)$. 
More formally, given $T \in \cert(C)$, we let $(\y^T,\x^T)$ to be the point such that for every $e=(g,g') \in \edges$, $y^T_e = 1$ if $g \in T$ and $g' \in T$ and $y_e^T = 0$ otherwise. Moreover, we let $x_v^T = \tau_T(v)$.  We claim that a point $(\y,\x)$ is in $\P_{x,y}(C) \cap \Z^{n_x + n_y}$ if and only if there exists a certificate $T$ such that $(\y,\x) = (\y^T,\x^T)$. Given a certificate $T$, showing that $(\y^T,\x^T) \in \P_{x,y}(C) \cap \Z^{n_x+n_y}$ is a simple check that every constraint in $\S_{x,y}(C)$ are satisfied. All constraints are straightforward to check but the one from \cref{eq proj}. Let $v \in V$. 
By definition, $x_v^T = \tau_T(v)$, that is, $x_v^T$ is set to $1$ if and only if an input gate labeled with literal $v$ appears in $T$ and $0$ otherwise. First, assume that literal $v$ does not appear in $T$. Then $y^T_e = 0$ for every $e \in \edges[C][v]$ and $\tau_T(v) = 0$ by definition. Hence $x^T_v = \tau_T(v) = \sum_{e \in \edges[C][v]} y^T_e$.
Now if literal $v$ appears in $T$, we have to show that there is exactly one edge $e \in \edges[C][v]$ such that $y^T_e = 1$. There is obviously at least one by definition of certificates. The uniqueness follows from the fact that certificates are trees directed to the output of the DNNF. 

For the other way around, consider $(\y,\x) \in \P_{x,y}(C) \cap \Z^{n_x + n_y}$. First, observe that for every gate $g$ and edge $e \in \ine[g]$, by (\ref{eq TDI system-or}) if $g$ is a $\vee$-gate or by (\ref{eq TDI system-and}) if $g$ is a $\wedge$-gate, we have:
\begin{equation}
  \label{eq:flow} y_e \leq \sum_{f \in \oute[g]} y_f.
\end{equation}
We now prove that for every gate $g$, $\sum_{f \in \oute[g]} y_f \in \{0,1\}$. This can be proven by induction on the depth of $g$, that is, the length of the longest path from $g$ to $o$. If the depth of $g$ is $1$, it means that $e=(g,o)$ is an edge of $C$ and is the only edge going out of $g$. Indeed, if there is another edge $(g,g')$ in $\oute[g]$, then since there is a path from $g'$ to $o$, there would be a path from $g$ to $o$ of length greater than $1$, contradicting the fact that $g$ has depth $1$. In other words, $\sum_{f \in \oute[g]} y_f = y_e$.  Now by (\ref{eq TDI system-o}), $y_e \leq 1$ which implies $\sum_{f \in \oute[g]} y_f = y_e \in \{0,1\}$ since $y_e$ is an integer. Now assume that the induction hypothesis holds for every gate of depth at most $d$ and let $g$ be a gate of depth $d+1$. Let $e = (g,g') \in \oute[g]$. Clearly, $g'$ has depth at most $d$, hence the induction hypothesis holds. In other words, $\sum_{f \in \oute[g']} y_f \in \{0,1\}$. But, $e \in \ine[g']$, hence by (\ref{eq:flow}), we have $y_e \leq \sum_{f \in \oute[g']} y_f$. In other words, $y_e \in \{0,1\}$. It remains to show that there is at most one edge $e$ in $\oute[g]$  such that $y_e > 0$. Assume toward a contradiction that two such edges $e,e' \in \oute[g]$ exist, that is, $y_e=y_{e'}=1$. By (\ref{eq:flow}) and by induction, if a gate $g'$ with depth at most $d$ has an incoming edge $f$ with $y_f=1$, then it has an outgoing edge $f'$ with $y_{f'}=1$. Hence, from $g$, we can define two paths going toward the output $o$, one starting with $e$ and the other with $e'$, and both following only edges $f$ with $y_f=1$. Both paths have to meet at some gate $g''$, with two incoming edges $e_1 = (g_1,g'')$ and $e_2 = (g_2,g'')$ with $y_{e_1} = y_{e_2} = 1$. Now observe that $g''$ cannot be a $\wedge$-gate, otherwise both input of $g''$ would share a variable (through the subcircuit rooted in $g$), contradicting the decomposability assumption. Hence $g''$ is a $\vee$-gate having of depth at most $d$ and having two incoming edges with $y_{e_1} = y_{e_2} = 1$. But by (\ref{eq TDI system-or}), it would imply $\sum_{f \in \oute[g'']} y_f \geq 2$, which contradicts the induction hypothesis if $g''$ is not the output of $C$. If $g''=o$, it violates  \eqref{eq TDI system-o}.
Hence $g$ has at most one outgoing edge $e$ with $y_e=1$ which establishes the induction at depth $d+1$.

We have established that for every $g$, $\sum_{f \in \oute[g]} y_f \in \{0,1\}$. One direct consequence is that $y_e\in\{0,1\}$ for every edge $e$ and $x_v \in \{0,1\}$ for every $v \in V$ by (\ref{eq proj}). Now, consider the subset $T$ of gates in $C$ defined as $g \in T$ if there is some edge $e$ of $C$ containing $g$ and such that $y_e = 1$. We prove that $T$ is a certificate of $C$ which would directly establish that $(\y,\x) = (\y^T,\x^T)$. It is enough to check that every condition of certificate are respected by $T$. First, let $g$ be a $\vee$-gate in $T$. By definition, it is in at least one edge $e$ with $y_e = 1$. By (\ref{eq TDI system-or}), $\sum_{f \in \ine[g]} y_f = \sum_{f \in \oute[g]} y_f$ and by what precedes the RHS of this equality must be equal to $1$. In other words, there is exactly one ingoing edge $e_1$ in $g$ and exactly one outgoing edge $e_2$ from $g$ such that $y_{e_1}=y_{e_2}=1$. This proves the first condition of certificate. 
Now, if $g$ is a $\wedge$-gate, again, at most one edge $e$ containing $g$ is such that $y_e=1$ by definition of $T$. From what precedes, we hence must have $\sum_{f \in \oute[g]} y_f = 1$. Hence, by (\ref{eq TDI system-and}), every ingoing edge $e$ of $g$ is such that $y_e = 1$. It remains to show that the third condition holds. Assume toward a contradiction that there is a gate $g$ such that there is an edge $e=(g',g)$ with $y_e=1$ but such that for every outgoing edge $e'=(g,g'')$, $y_{e'}=0$. It directly contradicts (\ref{eq:flow}). Hence the third condition of certificate is respected which concludes the proof. 
\end{proof}

We now show that $\P_{x,y}(C)$ is integral. 
\begin{lemma}
  \label{lem:pxy-integral}
The polyhedron $\P_{x,y}(C)$ is integral.  
\end{lemma}

We prove \cref{lem:pxy-integral} by proving that $P_y(C)$ is integral. This will be enough thanks to the following observation:
\begin{lemma}
\label{lem from Py to Pxy}
Let $C$ be a NNF circuit.
If $\P_y(C)$ is integral, then so is $\P_{x,y}(C)$.
\end{lemma}
\begin{proof}
Let $(\x, \y)$ be a vertex of $\P_{x,y}(C)$.
Then there exist $n_x+n_y$ linearly independent constraints among $\S_{x,y}(C)$ such that $(\x, \y)$ is the unique vector that satisfies these $n_x+n_y$ constraints at equality.
Note that constraints \eqref{eq proj} are the only ones containing variables $x$, and each such constraint contains exactly one $x$ variable with nonzero coefficient.
Thus, $n_x$ of the linearly independent constraints defining $(\x, \y)$ must be \eqref{eq proj}.
The remaining $n_y$ linearly independent constraints defining $(\x, \y)$ must be in $\S_y(C)$, and these constraints only involve $y$ variables.
Hence, $\y$ is a vertex of $\P_y(C)$. 
\end{proof}

To show that $\P_y(C)$ is an integral polyhedron we use total dual integrality.
A rational system of linear inequalities $Ay = b$, $y \ge 0$ is \emph{totally dual integral}, abbreviated \emph{TDI,} if the minimum in the LP-duality equation
$$
\max \{ c^\transp y : Ay=b, \ y \ge 0 \} = \min \{z^\transp b : zA \ge c\}
$$
has an integral optimum solution $z$ for each integral vector $c$ for which the minimum is finite.

\begin{theorem}
\label{th TDI}
Let $C$ be a DNNF circuit.
Then, the system $\S_y(C)$ is TDI and the polyhedron $\P_y(C)$ is integral.
\end{theorem}

\begin{proof}
It suffices to prove that the system $\S_y(C)$ is TDI.
In fact, since the right-hand-side is integral, it then follows from \cite{EdmGil77} (see also Corollary 22.1c in \cite{SchBookIP}) that $\P_y(C)$ is integral.
Hence, in the remainder of the proof we show that the system $\S_y(C)$ is TDI.

\textbf{Dual problem.}
Let $c \in \Z^{\edges}$.
To write the dual of the LP problem $\max \{ c^\transp y : y \text{ satisfies } \S_y \}$, we associate variable $z_o$ to constraint \eqref{eq TDI system-o}, 
variables $z_g$, for every $\vee$-gate $g$, to constraints \eqref{eq TDI system-or},
variables $z_{g,e}$, for every $\wedge$-gate $g$ and $e \in \ine[g]$, to constraints \eqref{eq TDI system-and}.
It will be useful to define, for every gate $g$, $\zvar[g]$ as the set of $z$ variables corresponding to $g$.
Note that sets $\zvar[g]$, for every gate $g$, partition all the $z$ variables.
Furthermore, if $g$ is an input gate we have $\zvar[g] = \emptyset$, if $g$ is a $\vee$-gate we have $\zvar[g] = \{z_g\}$, and if $g$ is a $\wedge$-gate we have $\zvar[g] = \{z_{g,e} : e \in \ine[g]\}$.
The dual of the LP problem $\max \{ c^\transp y : y \text{ satisfies } \S_y \}$ is 
\begin{align*}
\min && z_o \\
\text{s.t.} && z_h & \ge c_e && \qquad\text{$\forall e=(g,h) \in \edges$ with $g$ input, $h$ $\vee$-gate,} \\
&& z_{h,e} & \ge c_e && \qquad\text{$\forall e=(g,h) \in \edges$ with $g$ input, $h$ $\wedge$-gate,} \\
&& -z_g + z_h & \ge c_e && \qquad\text{$\forall e=(g,h) \in \edges$ with $g$ $\vee$-gate, $h$ $\vee$-gate,} \\
&& -z_g + z_{h,e} & \ge c_e && \qquad\text{$\forall e=(g,h) \in \edges$ with $g$ $\vee$-gate, $h$ $\wedge$-gate,} \\
&& -\sum_{f \in \ine[g]} z_{g,f} + z_h & \ge c_e && \qquad\text{$\forall e=(g,h) \in \edges$ with $g$ $\wedge$-gate, $h$ $\vee$-gate,} \\
&& -\sum_{f \in \ine[g]} z_{g,f} + z_{h,e} & \ge c_e && \qquad\text{$\forall e=(g,h) \in \edges$ with $g$ $\wedge$-gate, $h$ $\wedge$-gate.}
\end{align*}

\textbf{Algorithm.}
Next, we give an algorithm that, as we will show later, constructs an optimal solution to the dual which is integral.
The algorithm recursively assigns values to the variables starting from the inputs, and proceeding towards the output.
One variable is assigned its value only when all variables of its inputs have already been assigned. 
More precisely, a variable in $\zvar[h]$ is assigned its value only when all variables $\zvar[g]$, for each $g \in \ing[h]$ have already been assigned.

At the very beginning, the algorithm considers the gates $h$ such that $\ing[h]$ contains only input gates, since $g$ is an input gate if and only if $\zvar[g] = \emptyset$.
As a warm up, we describe the algorithm in this simpler case.
If $h$ is a $\vee$-gate, there can be several constraints in the dual lower-bounding $z_h$:
\begin{align*}
z_h \ge c_e && \qquad\text{$\forall e \in \ine[h]$.}
\end{align*}
Thus we assign $z^*_h := \max \{ c_e : e \in \ine[h]\}$. 
If $h$ is a $\wedge$-gate, there is exactly one constraint in the dual lower-bounding each variable $z_{h,e}$, for $e \in \ine[h]$:
\begin{align*}
z_{h,e} \ge c_e. 
\end{align*}
Thus we assign $z^*_{h,e} := c_e$ for each $e \in \ine[h]$.

Next, we describe the general iteration of the algorithm.
Let $h$ be a gate such that, for each $g \in \ing[h]$, the variables in $\zvar[g]$ have already been assigned. 
Consider first the case where $h$ is a $\vee$-gate.
The constraints in the dual lower-bounding $z_h$ are:
\begin{align}
\label{eq or lb}
\begin{split}
z_h \ge c_e & \qquad\text{$\forall e=(g,h) \in \edges$ with $g$ input,} \\
z_h \ge c_e + z_g & \qquad\text{$\forall e=(g,h) \in \edges$ with $g$ $\vee$-gate,} \\
z_h \ge c_e + \sum_{f \in \ine[g]} z_{g,f} & \qquad\text{$\forall e=(g,h) \in \edges$ with $g$ $\wedge$-gate.}
\end{split}
\end{align}
Note that all variables on the right-hand of the above constraints have already been assigned by the algorithm.
Thus we assign $z^*_h$ as follows:
\begin{align}
\label{eq or max}
z^*_h := \max
\begin{Bmatrix*}[l]
c_e & \text{$\forall e=(g,h) \in \edges$ with $g$ input,} \\
c_e + z^*_g & \text{$\forall e=(g,h) \in \edges$ with $g$ $\vee$-gate,} \\ 
c_e + \sum_{f \in \ine[g]} z^*_{g,e} & \text{$\forall e=(g,h) \in \edges$ with $g$ $\wedge$-gate}
\end{Bmatrix*}.
\end{align}
Next, consider the case where $h$ is a $\wedge$-gate, and let $e=(g,h) \in \ine[h]$.
There is exactly one constraint in the dual lower-bounding $z_{h,e}$, which is:
\begin{align}
\label{eq and lb}
\begin{split}
z_{h,e} \ge c_e & \qquad\text{if $g$ input,} \\
z_{h,e} \ge c_e + z_g & \qquad\text{if $g$ $\vee$-gate,} \\
z_{h,e} \ge c_e + \sum_{f \in \ine[g]} z_{g,f} & \qquad\text{if $g$ $\wedge$-gate.}
\end{split}
\end{align}
Also in this case, all variables on the right-hand of the above constraints have already been assigned by the algorithm.
Thus we assign $z^*_{h,e}$ as follows:
\begin{align}
\label{eq and max}
z^*_{h,e} := 
\begin{cases}
c_e & \qquad\text{if $g$ input,} \\
c_e + z^*_g & \qquad\text{if $g$ $\vee$-gate,} \\
c_e + \sum_{f \in \ine[g]} z^*_{f,e} & \qquad\text{if $g$ $\wedge$-gate}.
\end{cases}
\end{align}
This concludes the description of the algorthm. Note that, since $c$ is integral, $z^*$ is integral.

\textbf{Termination.}
Our algorithm assigns a value to each variable.
This follows from the structure of the directed graph $C$, since $C$ is acyclic and the inputs of $C$ are precisely the gates of $C$ that have no ingoing edge. Moreover $C$ is a connected graph since we assumed there is a path from every gate $g$ to the output $o$.

\textbf{Feasibility.}
It is simple to see from the definition of our algorithm that the solution $z^*$ constructed is feasible to the dual.
In fact, in each step, the component of $z^*$ considered is defined in a way that it satisfies all inequalities lower-bounding it.
On top of that, each constraint in the dual is considered at some point in the algorithm.

\textbf{Optimality.}
To prove that the solution $z^*$ constructed by the algorithm is optimal to the dual, we show that any $z$ feasible for the dual must have objective value greater than or equal to that of $z^*$, which is $z^*_o$.
To do so, we give a procedure that selects a number of constraints of the dual.
Summing together the selected constraints yields the inequality $z_o \ge z^*_o$.
We remark that the procedure is strongly connected to the concept of certificates.

The procedure is recursive and starts from the gate $o$, which in particular is a $\vee$ gate. We give the general recursive construction for a gate $h$.
\begin{itemize}
\item
If $h$ is an input gate, there is nothing to do.
\item
If $h$ is a $\vee$ gate, we select one constraint among \eqref{eq or lb} that achieves the maximum in \eqref{eq or max}.
This essentially amounts to selecting an edge $e=(g,h) \in \edges$, and therefore a new gate $g$.
We then apply recursively the construction to the gate $g$.
\item
If $h$ is a $\wedge$ gate, we select, for every $e \in \ine[h]$, the one constraint \eqref{eq and lb}.
This essentially amounts to selecting all edges $e=(g,h) \in \edges$, and therefore all new gates $g \in \ing[h]$.
We then apply recursively the construction to all the gates $g \in \ing[h]$.
\end{itemize}
It is simple to check that summing all the inequalities selected by the above recursive procedure yields $z_o \ge z^*_o$.
\end{proof}

\begin{observation}
  \label{ex TU}
  A common way of showing integrality is by showing that the matrix of the system is totally unimodular. This method would fail in our case however and this is why we resorted to TDI. Figure~\ref{fig:notTU} depicts a DNNF $C$ such that $\S_y(C)$ is:

  \begin{subequations}
    \label{eq exTU}
    \begin{align}
      y_1 + y_6 = 1
      \label{eq exTU-1} \\
      y_2 + y_4 - y_1 = 0
      \label{eq exTU-2} \\
      y_3 - y_2 = 0
      \label{eq exTU-3} \\
      y_5 - y_3 = 0
      \label{eq exTU-4} \\
      y_7 - y_3 = 0
      \label{eq exTU-5} \\
      y_8 - y_4 - y_5 = 0
      \label{eq exTU-6} \\
      y_9 - y_6 - y_7 = 0 \\
      \label{eq exTU-6-1} 
      y_{10} - y_6 - y_7 = 0 \\
      \label{eq exTU-7}
      0 \le y \le 1.
    \end{align}
  \end{subequations}
  The matrix obtained from the constraint matrix of the above system $\S_y(C)$ by dropping the row corresponding to constraint \eqref{eq exTU-3}, \eqref{eq exTU-6-1} and the columns corresponding to variables $y_2,y_8,y_9,y_{10}$ is given by 
$$
\begin{pmatrix}
+1 & 0 & 0 & 0 & +1 & 0 \\
-1 & 0 & 1 & 0 & 0 & 0 \\
0 & -1 & 0 & 1 & 0 & 0 \\
0 & -1 & 0 & 0 & 0 & 1 \\
0 & 0 & -1 & -1 & 0 & 0 \\
0 & 0 & 0 & 0 & -1 & -1
\end{pmatrix}.
$$
This is a square matrix with determinant equal to $2$.
Therefore, the constraint matrix of the system $\S_y(C)$ is not totally unimodular.

  \begin{figure}
    \centering
    \includegraphics[scale=1]{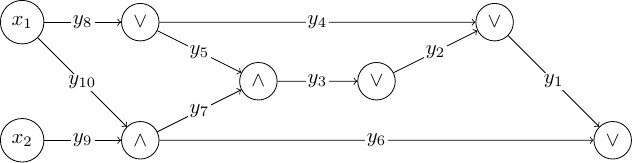}
    \caption{A smooth DNNF circuit $C$ where $\S_y(C)$ is not totally unimodular.}
    \label{fig:notTU}
\end{figure}

\end{observation}

As a direct consequence of the results presented in this section, we obtain extended formulations of polynomial size for the multilinear polytope $\MP_H$, and for the convex hull of the points in the multilinear set $\MS_H$ that satisfy extended cardinality constraints, provided that $H$ is $\beta$-acyclic, or the incidence treewidth of $H$ is bounded by $\log(\poly(|V|,|E|))$. Indeed, by \cref{thm:bpo-tw-dnnf,thm:bpo-beta-dnnf}, the exists a polynomial size DNNF representing the multilinear set of such BPO instances, and the extended formulation is extracted from the DNNF using \cref{thm:extended}. Observe that this also applies to $\beta$-acyclic and bounded (incidence) treewidth BPO instances with literals and with extended cardinality constraints since polynomial size DNNF of the multilinear set also exist in this case using the results from \cref{sec:beyond}.


\section{Experiments}
\label{sec:experiments}

The connection between BPO instances and weighted Boolean functions from \cref{thm:bpotobop} suggests that one could leverage tools initially developed for Boolean function into solving BPO instances. In this section, we compare such an approach with the SCIP solver. We modified the d4 knowledge compiler~\cite{lagniez2017improved} (available at \url{https://github.com/crillab/d4v2}) so that it directly computes the optimal value of a Boolean function given as a CNF $F$ on a given weighting $w$. 

We tested our approach on the Low Autocorrelation Binary Sequences (LABS problem for short) which has been defined in~\cite{liers2010non} and can be turned into instances of the BPO problem. We run our experiment on a 13th Gen Intel(R) Core(TM) i7-1370P with 32Go of RAM and with a 60 minutes timeout. We used the instances available at on the MINLPLib library~\cite{bussieck2003minlplib}\footnote{Available in PIP format at \url{https://www.minlplib.org/applications.html\#AutocorrelatedSequences}} and we compare the performances with SCIP 8.0.4~\cite{scip}, GUROBI~\cite{gurobi} and CPLEX~\cite{cplex}.
These instances have two parameters, $n$ and $w$, and are reported under the name \texttt{bernasconi.n.w} in our experiments.

%

%
Our results are reported on \cref{fig:results} and show that the approach, on this particular instances, runs faster than these solvers by two orders of magnitude at least. The optimal value reported by each tool when they terminate matches the one given on MINLPLIB. We believe that this behavior is explained by the fact that the treewidth of \texttt{bernasconi.n.w} is of the order of $w$ and that d4 is able to leverage small treewidth to speed up computation, an optimization the linear solver are not doing, leading to timeouts even for small values of $w$ . 
 
\begin{figure}
  \centering
  \input{results.tex}
\caption{Runtimes (in seconds) of different solvers on LABS instances. A dash (--) denotes that the solver was not able to finish computation in under an hour.}
\label{fig:results}
\end{figure}

\smallskip
\textbf{Funding and Acknowledgement.} A. Del Pia is partially funded by AFOSR grant FA9550-23-1-0433. Any opinions, findings, and conclusions or recommendations expressed in this material are those of the authors and do not necessarily reflect the views of the Air Force Office of Scientific Research. F. Capelli is partially funded by project KCODA from Agence Nationale de la Recherche, project ANR-20-CE48-0004. F. Capelli is grateful to Jean-Marie Lagniez who helped him navigate the codebase of d4 to modify it in order to solve \BOP{} problems on d-DNNF.

\bibliographystyle{plain}
\bibliography{biblio}

\appendix

\section{Graph Theoretical Lemmas}
\label{appendix graphs}

Let $G=(V,E)$ be a graph and $E' \subseteq E$. The \emph{edge subdivision of $G$ along $E'$} is a graph $sub(G,E')$ obtained from $G$  by splitting every edge of $E'$ in two. More formally, $sub(G,E') = (V',E'')$ where $V' = V \cup \{v_{e} \mid e \in E'\}$ and $E'' = (E \setminus E') \cup \{\{v, v_{e}\} \mid e \in E', v \in e \}$. One can easily see that splitting edges does not change the treewidth of $G$:

\begin{lemma}
\label{lem:twsub}
For every graph $G=(V,E)$, $E' \subseteq E$ and $G' = sub(G,E')$. We have $\tw{G'} \leq \tw{G}$.
\end{lemma}
\begin{proof}
  If $G$ is a tree, then it is clear that $G'$ is also a tree and automatically we have that $\tw{G}=\tw{G'}=1$. 
  Otherwise let $T$ be a tree decomposition of $G$ of width $k>1$. We transform it into a decomposition $T'$ of $G'$ as follows: for every
$e=\{u,v\} \in E'$, let $B_e$ be a bag of $T$ such that $e \subseteq B_e$ (note that it exists by definition of tree decomposition). 
In $T'$, we attach a new bag $B=\{u,v,v_e\}$ to $B(e)$, for all $e \in E'$. 
This is clearly a tree decomposition of $G'$ and its width is the same as the width of $T$ since we only added bags of size $3$ and the treewidth of $T$ is bigger than $1$.

\end{proof}

Let  $G=(V,E)$ be a graph. 
For $u \in V$, the \emph{neighborhood of $u$} is defined as $\neigh[u][G] = \bigcup_{e \in E , u \in e} e$.
A \emph{module $M \subseteq V$ of $G$} is a set of vertices such that for every $u,v \in M$, $\neigh[u][G] \setminus M = \neigh[v][G] \setminus M$, that is, every vertex of $M$ has the same neighborhood outside of $M$. We denote by $\oneigh[M][G]$ this set. A \emph{partition of $G$ into modules} is a partition $\calM=(M_1, \dots, M_k)$ of $V$ where $M_i$ is a module of $V$ (possibly of size $1$). The \emph{modular contraction of $G$ wrt $\calM$}, denoted by $G/\calM$, is the graph whose vertices are $\calM$ and such that there is an edge between $M_i$ and $M_j$ if and only if $\oneigh[M][G] \cap M' \neq \emptyset$ (observe that in this case $M' \subseteq \oneigh[M][G]$ and $M \subseteq \oneigh[M'][G]$). While it is known that the treewidth of a modular contraction of $G$ might decrease it arbitrarily~\cite{PaulusmaSS13}, we can bound it as follows:

\begin{lemma}
\label{lem:twmodule}
    Let $G=(V,E)$ be a graph, $\calM=\{M_1,\dots,M_k\}$ a partition of $G$ into modules and $m=\max_{i=1}^k |M_i|$. Then $\tw{G} \leq m\times (1+\tw{G/\calM})-1$.
\end{lemma}
\begin{proof}
    We prove it by transforming any tree decomposition $T$ of $G/\calM$ of width $k$ into a tree decomposition of $G$ of width at most $(k+1)m-1$. To do that, we construct $T'$ by replacing every occurrence of $M_i$ in $T$ by the vertices of $M_i$. If $B$ is a bag of $T$, we denote by $B'$ its corresponding bag in $T'$. Each bag $B'$ hence contains at most $(\max |M_i|)\cdot(k+1) = m(k+1)$ and the bound on the width of $T'$ follows. It remains to prove that $T'$ is indeed a tree decomposition of $G$. 
    
    We start by proving that every edge of $G$ is covered in $T'$. Let $\{u,v\}$ be an edge of $G$ and $M_u, M_v \in \calM$ be such that $u \in M_u$ and $v \in M_v$. If $M_u=M_v$ then edge $\{u,v\}$ is covered by $B'$ in $T'$ for any bag $B$ of $T$ that contains $M_u$. Otherwise, since $\{u,v\}$ is an edge of $G$, we have $u \in \oneigh[M_v][G]$. Hence $\{M_u,M_v\}$ is an edge $G/\calM$. Hence there is a bag $B$ of $T$ such that $M_u,M_v \in B$. Hence edge $\{u,v\}$ is covered in $T'$ by $B'$. 
    
    Finally,  we prove that every vertex of $G$ is connected in $T'$. Indeed, let $u$ be a vertex of $G$ such that $u \in B_1' \cap B_2'$. It means that there exists $M_1 \in B_1$ and $M_2 \in B_2$ such that $u \in M_1$ and $u \in M_2$. Now, by definition, $u$ is in exactly one module, that is, $M_1=M_2=M$. Hence, by connectedness of $T$, $M$ is in every bag on the path from $B_1$ to $B_2$ in $T$ which means that $u$ is in every bag on the path from $B_1'$ to $B_2'$.
\end{proof}

\end{document}

%% file: results.tex
\begin{longtable}[]{@{}l|l|l|l|l@{}}
\toprule
name & d4 & scip & cplex & gurobi \\
\midrule
\endhead
bernasconi.20.3 & < 0.01 & 0.01 & 0.01 & 0.01 \\
bernasconi.20.5 & < 0.01 & 6.77 & 0.85 & 0.65 \\
bernasconi.20.10 & 0.2 & 90.22 & 14.86 & 4.8 \\
bernasconi.20.15 & 3.98 & 292.31 & 50.14 & 58.18 \\
bernasconi.25.3 &  0.01 & 0.02 & 0.01 &  0.01 \\
bernasconi.25.6 & 0.06 & 86.82 & 50.79 & 10.18 \\
bernasconi.25.13 & 3.45 & 1240.47 & 355.23 & 274.34 \\
bernasconi.25.19 & 137.24 & -- & 1223.65 & 1045.32 \\
bernasconi.25.25 & 1570.98 & -- & -- & -- \\
bernasconi.30.4 & 0.02 & 28.07 & 4.81 & 2.13 \\
bernasconi.30.8 & 1.4 & 2828.35 & 2465.93 & 549.56 \\
bernasconi.30.15 & 77.09 & -- & -- & -- \\
bernasconi.30.23 & -- & -- & -- & -- \\
bernasconi.30.30 & -- & -- & -- & -- \\
bernasconi.35.4 & 0.01 & 44.06 & 10.19 & 9.23 \\
bernasconi.35.9 & 7.13 & -- & -- & -- \\
bernasconi.35.18 & 688.73 & -- & -- & -- \\
bernasconi.35.26 & -- & -- & -- & -- \\
bernasconi.35.35 & -- & -- & -- & -- \\
bernasconi.40.5 & 0.09 & 2345.71 & 1503.44 & 320.75 \\
bernasconi.40.10 & 44.86 & -- & -- & -- \\
bernasconi.40.20 & -- & -- & -- & -- \\
bernasconi.40.30 & -- & -- & -- & -- \\
bernasconi.40.40 & -- & -- & -- & -- \\
bernasconi.45.5 & 0.16 & -- & -- & 669.28 \\
bernasconi.45.11 & 217.43 & -- & -- & -- \\
bernasconi.45.23 & -- & -- & -- & -- \\
bernasconi.45.34 & -- & -- & -- & -- \\
bernasconi.45.45 & -- & -- & -- & -- \\
bernasconi.50.6 & 1.07 & -- & -- & -- \\
bernasconi.50.13 & -- & -- & -- & -- \\
bernasconi.55.6 & 1.61 & -- & -- & -- \\
bernasconi.60.8 & 46.11 & -- & -- & -- \\
bernasconi.60.15 & -- & -- & -- & -- \\
\bottomrule
\end{longtable}